\documentclass[english]{revcoles}
\usepackage{amsfonts,graphics,epsf,graphicx}
\usepackage{subfigure,rotating}

\newtheorem{theorem}{Theorem}[section]
\newtheorem{proposition}{Proposition}

\newtheorem{remark}{Remark}[section]

\begin{document}
\title[maintitle = The Beta-Gompertz Distribution,
       secondtitle = La Distribución Beta - Gompertz,
       shorttitle = The Beta-Gompertz Distribution
]
\begin{authors}
\author[firstname = Ali Akbar,
        surname = Jafari,
        numberinstitution = 1,
        affiliation = Corresponding,
        email = aajafari@yazd.ac.ir]
\author[firstname = Saeid,
        surname = Tahmasebi,
        numberinstitution = 2,
        email = tahmasebi@pgu.ac.ir
        ]
\author[firstname = Morad,
        surname = Alizadeh,
        numberinstitution = 3,
        ]
\end{authors}
\begin{institutions}
     \institute[subdivision = Department of Statistics,
                institution = Yazd University,
                city = Yazd,
                country = Iran]
     \institute[subdivision = Department of Statistics,
                institution = Persian Gulf University,
                city = Bushehr,
                country = Iran]
     \institute[subdivision = Department of Statistics,
                institution = Ferdowsi  University of Mashhad,
                city = Mashhad,
                country = Iran]

\end{institutions}
\begin{mainabstract}
In this paper, we introduce  a new four-parameter generalized version of the Gompertz  model which is called Beta-Gompertz (BG) distribution.
It includes some well-known lifetime distributions such as beta-exponential and generalized Gompertz distributions  as special sub-models. This new distribution is quite flexible and can be used effectively in modeling survival data and reliability problems. It can have a decreasing, increasing,  and bathtub-shaped failure rate function depending on its parameters.  Some mathematical properties of the new distribution, such as closed-form expressions for the density, cumulative distribution, hazard rate function, the $k$th order moment, moment generating function, Shannon entropy, and the quantile measure  are provided. We discuss maximum likelihood estimation of the BG parameters from one observed sample and derive the observed Fisher's information matrix. A simulation study is performed in order to investigate this proposed estimator for parameters. At the end, in order to show the BG distribution flexibility, an application using a real data set is presented.
\end{mainabstract}
\begin{secondaryabstract}
En este artículo, se introduce una versión generalizada en cuatro parámetros de la distribución de Gompertz denominada como la distribución Beta-Gompertz (BG). Esta incluye algunas distribuciones de duración de vida bien conocidas como la beta exponencial y distribuciones Gompertz generalizadas como casos especiales. Esta nueva distribución es flexible y puede ser usada de manera efectiva en datos de sobrevida y problemas de confiabilidad. Su función de tasa de falla puede ser decreciente, creciente o en forma de bañera dependiendo de sus parámetros. Algunas propiedades matemáticas de la distribución como expresiones en forma cerrada para la densidad, función de distribución, función de riesgo, momentos k-ésimos, función generadora de momentos, entropía de Shannon y cuantiles son presentados. Se discute la estimación máximo verosímil de los parámetros desconocidos del nuevo modelo para la muestra completa y se obtiene una expresión para la matriz de información. Con el fin de mostrar la flexibilidad de esta distribución, se presenta una aplicación con datos reales. Al final, un estudio de simulación es desarrollado.
\end{secondaryabstract}
\section{Introduction}
\label{sec.int}
The Gompertz (G) distribution is a flexible distribution that can be skewed to the right and to the left. This distribution is a generalization of the  exponential (E) distribution and is commonly used in many applied problems, particularly in lifetime data analysis
 \citeaffixed[p. 25]{jo-ko-ba-95-2}{}. The G distribution is considered for the analysis of survival, in some sciences such as gerontology \cite{br-fo-74}, computer \cite{oh-ok-do-09}, biology \cite{economos-82},  and marketing science \cite{be-gl-12}. The hazard rate function (hrf) of G distribution is an increasing function and often applied to describe the distribution of adult life spans by actuaries and demographers \cite{wi-ko-00}.
The G distribution with parameters $\theta>0$ and $\gamma>0$  
 has the cumulative distribution function (cdf)
\begin{eqnarray}\label{eq.G}
G(x)=1-e^{-\frac{\theta}{\gamma}(e^{\gamma x}-1)}, \;\;x\geq0, \;\; \beta>0, \;\;\gamma>0,
\end{eqnarray}
and the  probability density function (pdf)
\begin{eqnarray}\label{eq.g}
g(x)= \theta e^{\gamma x}e^{-\frac{\theta}{\gamma}(e^{\gamma x}-1)}.
\end{eqnarray}
This case is denoted by $X\sim \text{G}(\theta,\gamma)$.

Recently, a generalization based on the idea of \citeasnoun{gu-ku-99} was proposed by \citeasnoun{el-al-al-13}.
This new distribution is known as generalized Gompertz (GG) distribution which  includes the E, generalized exponential (GE), and G distributions \cite{el-al-al-13}.

In this paper, we introduce a new generalization of G distribution
which results of the application of the G distribution to the beta generator proposed by \citeasnoun{eu-le-fa-02},
called the  beta-Gompertz (BG) distribution.
Several generalized distributions have been proposed under this methodology:
beta-Normal distribution \cite{eu-le-fa-02},  
beta-Gumbel distribution \cite{na-ko-04},     
beta-Weibull distribution \cite{fa-le-ol-05}, 
beta-exponential (BE) distribution, \cite{na-ko-06}, 
beta-Pareto distribution \cite{ak-fa-le-08},   
beta-modified Weibull distribution \cite{si-or-co-10}, 
beta-generalized normal distribution \cite{ci-re-co-na-12}. 
The BG distribution includes some well-known distribution: E distribution,
GE distribution \cite{gu-ku-99}, 
BE distribution \cite{na-ko-06},  
G distribution, GG distribution \cite{el-al-al-13}. 

This paper is organized as follows: In Section \ref{sec.def}, we define the density and failure rate functions and outline some special cases of the BG distribution. In Sections \ref{sec.pr} we provide some extensions and properties of the cdf, pdf, $k$th  moment and moment generating function of the BG distribution. Furthermore, in these sections, we derive corresponding expressions for the order statistics, Shannon entropy and quantile measure. In Section \ref{sec.es},
we discuss maximum likelihood estimation of the BG parameters from one observed sample and derive the observed Fisher's information matrix.
A simulation study is performed in Section \ref{sec.si}. Finally,
an application of the BG using a real data set is presented in
Section \ref{sec.ex}.

\section{The BG distribution}
\label{sec.def}
In this section, we introduce the four-parameter BG distribution. The idea of  this distribution rises  from the following general class:
If $G$ denotes the cdf of a random variable then a generalized class of distributions can be
defined by
\begin{equation}\label{eq.FG}
F(x)=I_{G(x)}(\alpha,\beta)=\frac{1}{B(\alpha,\beta)}\int_{0}^{G(x)}t^{\alpha-1}(1-t)^{\beta-1}dt,\;\; \alpha,\beta>0,
\end{equation}
 where $I_{y}(\alpha,\beta)=\frac{B_{y}(\alpha,\beta)}{B(\alpha,\beta)}$ is the incomplete beta function ratio and
 $B_{y}(\alpha,\beta)=\int_{0}^{y}t^{\alpha-1}(1-t)^{\beta-1}dt$ is the incomplete beta function.

Consider that $g(x)=\frac{d G(x)}{dx}$ is the density of the baseline distribution. Then the probability density function corresponding to (\ref{eq.FG}) can be written in the form
\begin{equation}
f(x)=\frac{g(x)}{B(\alpha,\beta)}[G(x)]^{\alpha-1}[1-G(x)]^{\beta-1}.
\end{equation}
We now introduce the BG distribution by taking $G(x)$ in (\ref{eq.FG}) to the cdf in  (\ref{eq.G}) of the G distribution. Hence, the pdf of BG can be written as
\begin{equation}\label{eq.fBG}
f(x)=\frac{\theta e^{\gamma x}e^{-\frac{\beta\theta}{\gamma}(e^{\gamma x}-1)}}{B(\alpha,\beta)}[1-e^{-\frac{\theta}{\gamma}(e^{\gamma x}-1)}]^{\alpha-1}.
\end{equation}
and we use the notation $X\sim BG(\theta,\gamma,\alpha,\beta)$.

\begin{theorem}
Let $f(x)$  be the pdf of the BG distribution. The limiting behavior of $f$ for different values of its parameters is given bellow:
\begin{itemize}
  \item[i.]  If $\alpha=1$ then ${\mathop{\lim }_{x \rightarrow 0^+}  f(x)}=\theta\beta.$
  \item[ii.] If $\alpha>1$ then ${\mathop{\lim }_{x \rightarrow 0^+}  f(x)}=0.$
  \item[iii.] If $0<\alpha<1$ then ${\mathop{\lim }_{x \rightarrow 0^+} f(x)}=\infty.$
  \item[iv.] ${\mathop{\lim }_{x \rightarrow \infty} f(x )}=0.$
\end{itemize}
\end{theorem}
\begin{proof}
The proof of parts (i)-(iii) are obvious. For part (iv), we have
$$
0\leq [1-e^{-\frac{\theta}{\gamma}(e^{\gamma x}-1)}]^{\alpha-1}<1 \Rightarrow 0<f(x)<\frac{\theta e^{\gamma x}e^{-\frac{\beta\theta}{\gamma}(e^{\gamma x}-1)}}{B(\alpha,\beta)}.
$$
It can be easily shown that
$${\mathop{\lim }_{x \rightarrow \infty} \theta e^{\gamma x}e^{-\frac{\beta\theta}{\gamma}(e^{\gamma x}-1)}}=0.$$
and the proof is completed.
\end{proof}

The hrf of BG distribution is given by
\begin{equation}
h(x)=\frac{\theta e^{\gamma x}e^{-\frac{\beta\theta}{\gamma}(e^{\gamma x}-1)}}{B(\alpha,\beta)-B_{G(x)}(\alpha,\beta)}[1-e^{-\frac{\theta}{\gamma}(e^{\gamma x}-1)}]^{\alpha-1}.
\end{equation}
Recently, it is observed \cite{gu-gu-07} 
 that the reversed hrf plays an important role in the reliability analysis. The reversed hrf of the $BG(\theta,\gamma,\alpha,\beta)$  is
\begin{equation}
r(x)=\frac{\theta e^{\gamma x}e^{-\frac{\beta\theta}{\gamma}(e^{\gamma x}-1)}}{B_{G(x)}(\alpha,\beta)}[1-e^{-\frac{\theta}{\gamma}(e^{\gamma x}-1)}]^{\alpha-1}.
\end{equation}

Plots of pdf and hrf function of the BG distribution for different values of its parameters are given in Figure \ref{fig.fGB} and Figure \ref{fig.hGB}, respectively.

\begin{figure}
\centering
\includegraphics[scale=0.27]{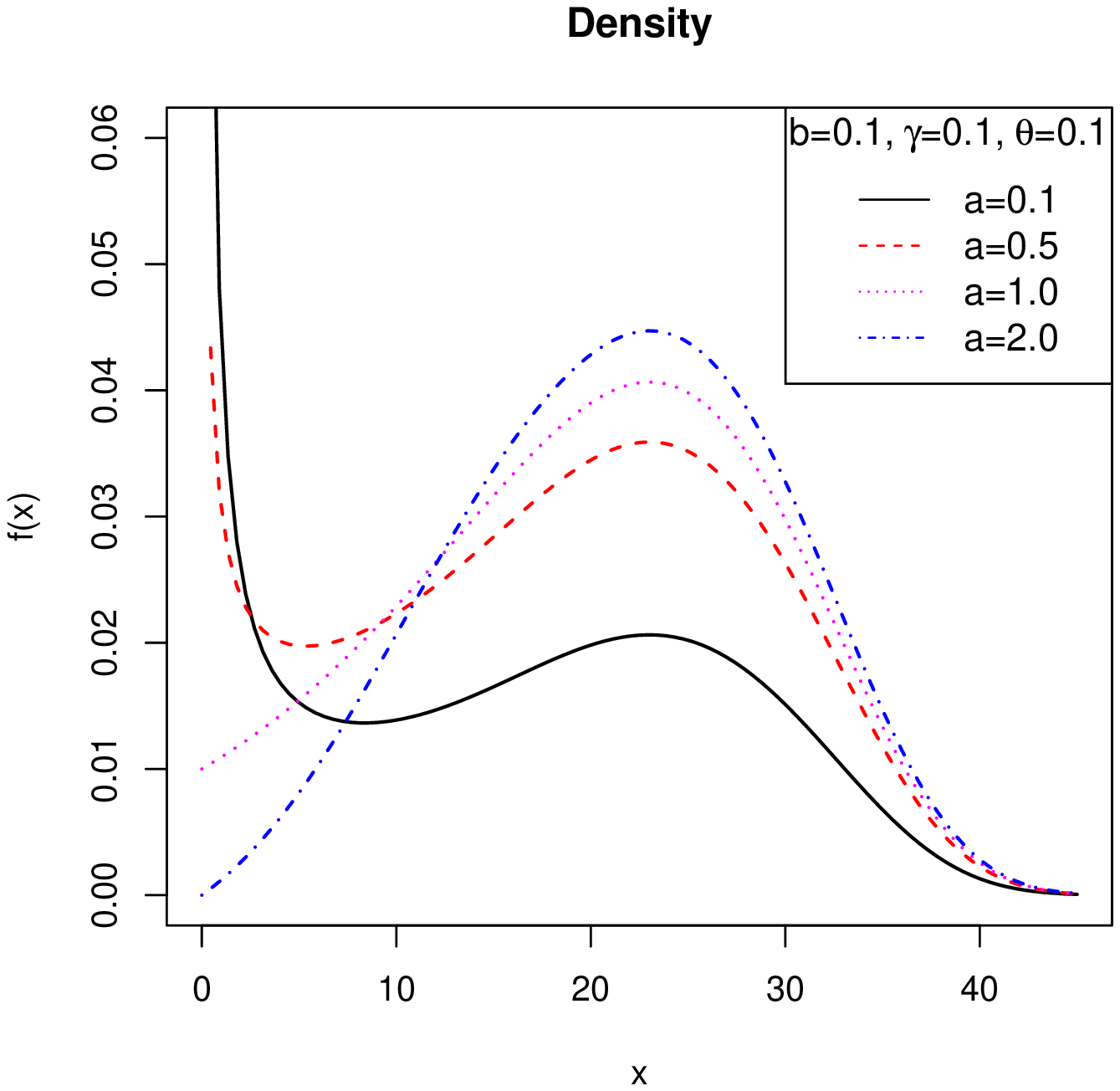}
\includegraphics[scale=0.27]{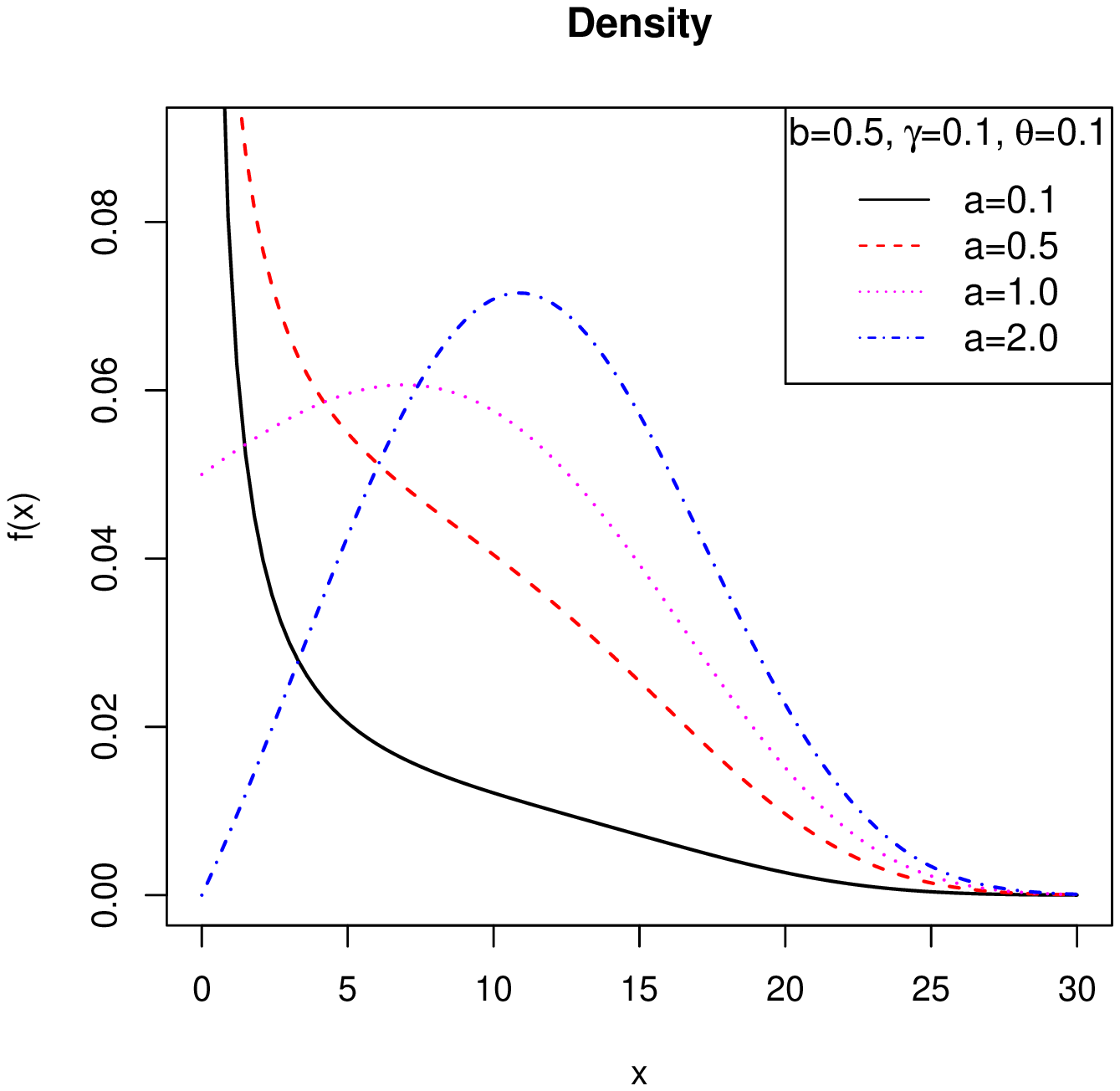}
\includegraphics[scale=0.27]{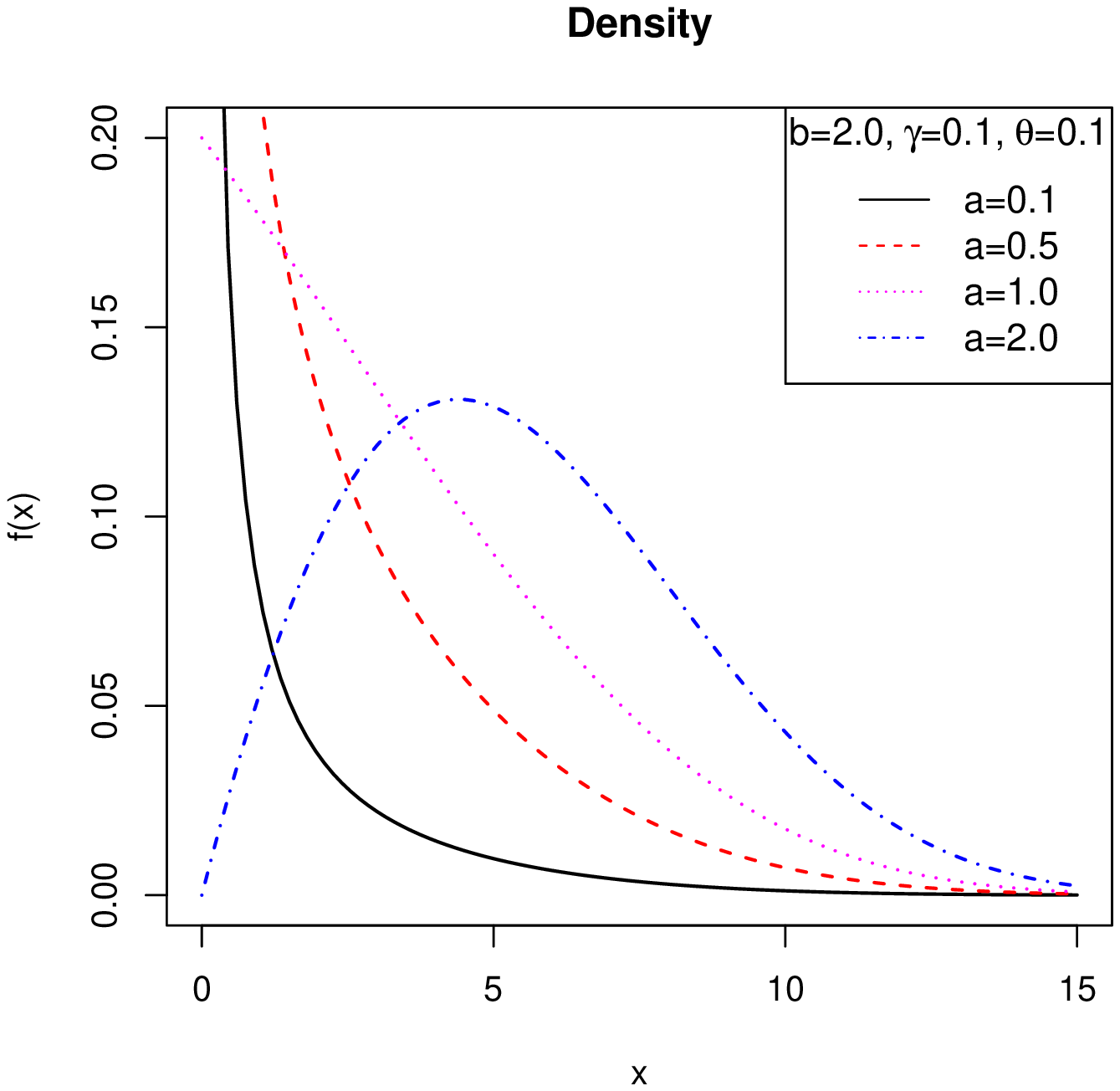}
\includegraphics[scale=0.27]{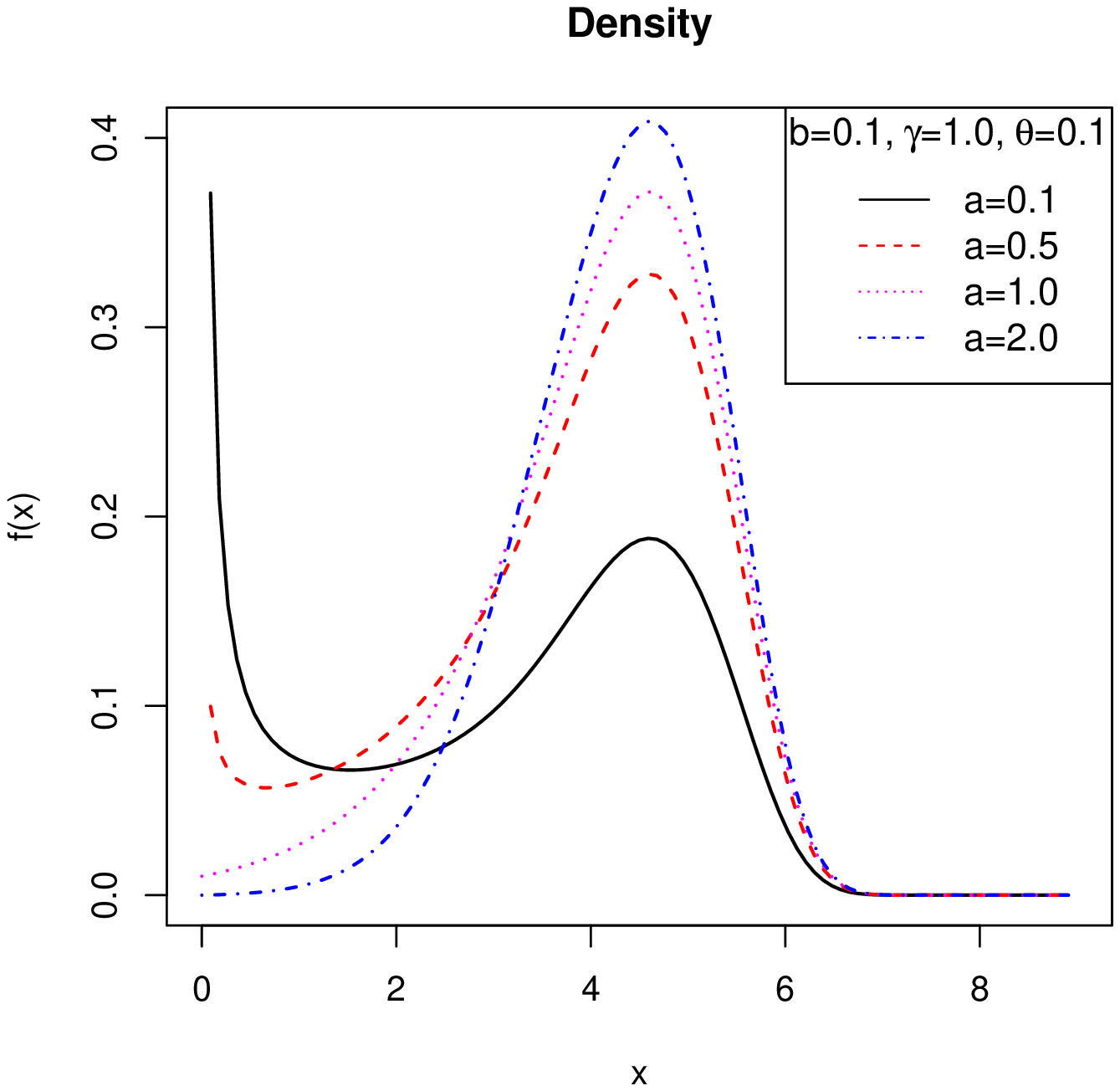}
\includegraphics[scale=0.27]{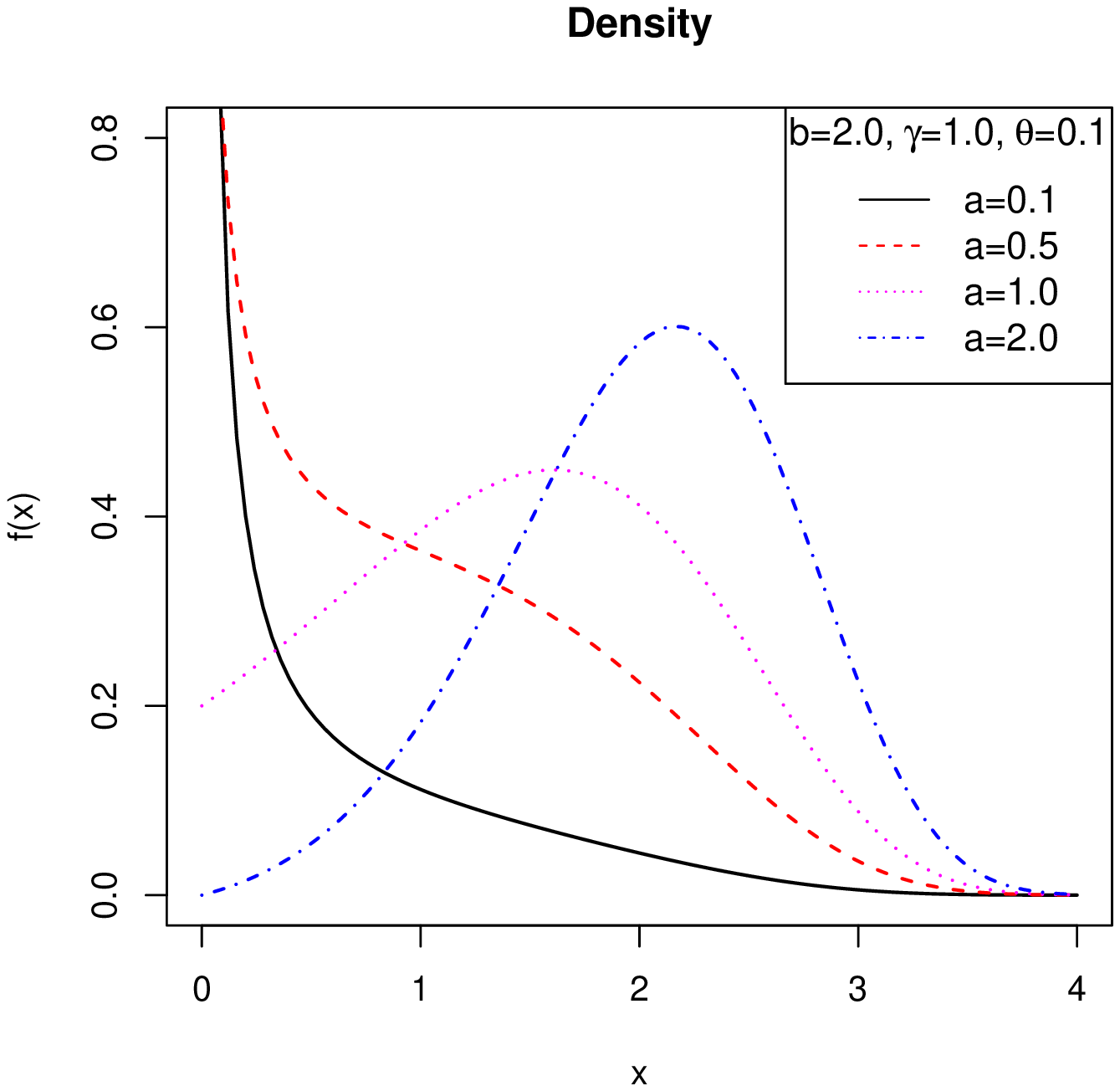}
\includegraphics[scale=0.27]{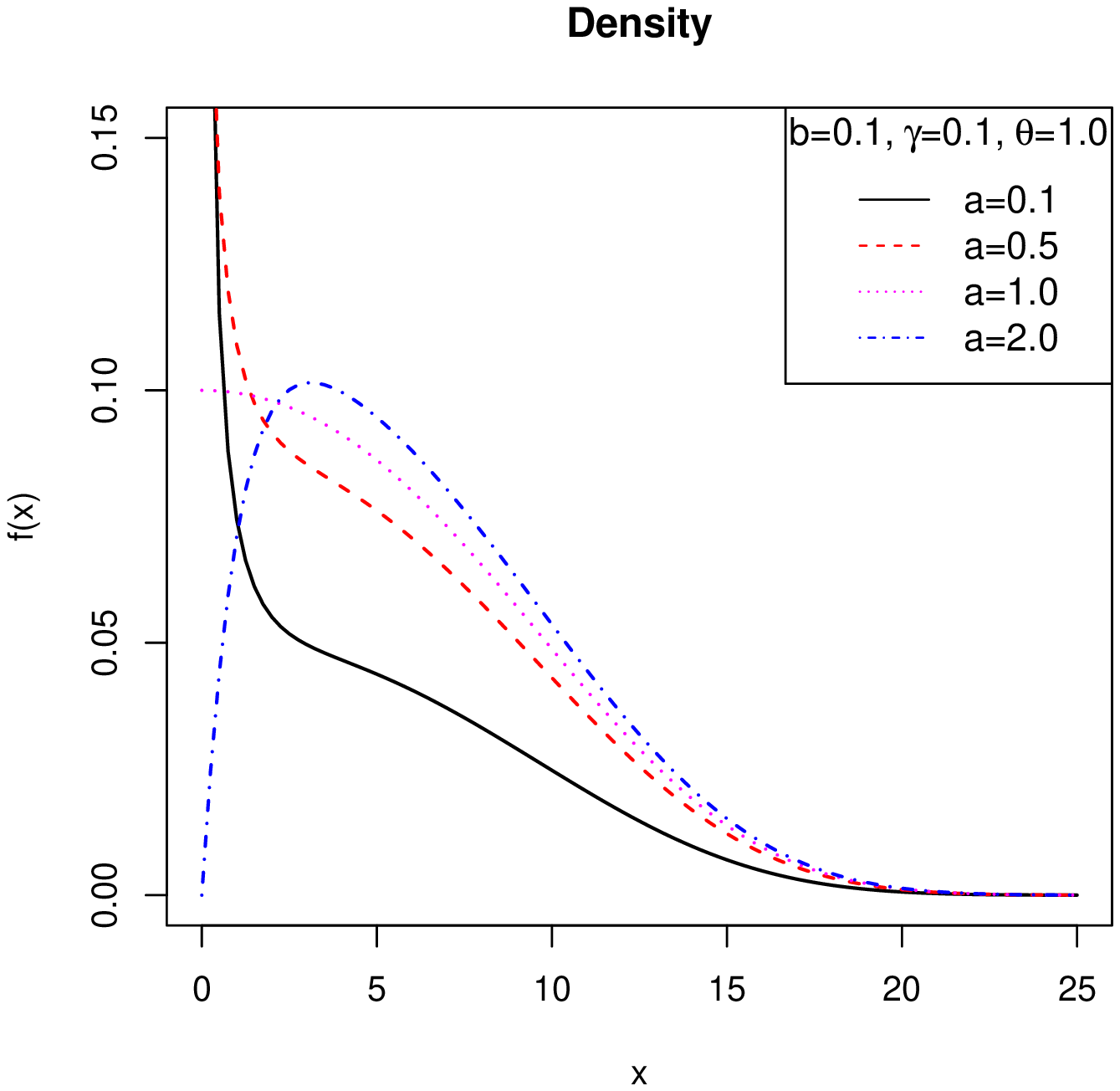}
\caption[]{Plots of density  functions of BG for different values of parameters.}\label{fig.fGB}
\end{figure}

\begin{figure}
\centering
\includegraphics[scale=0.27]{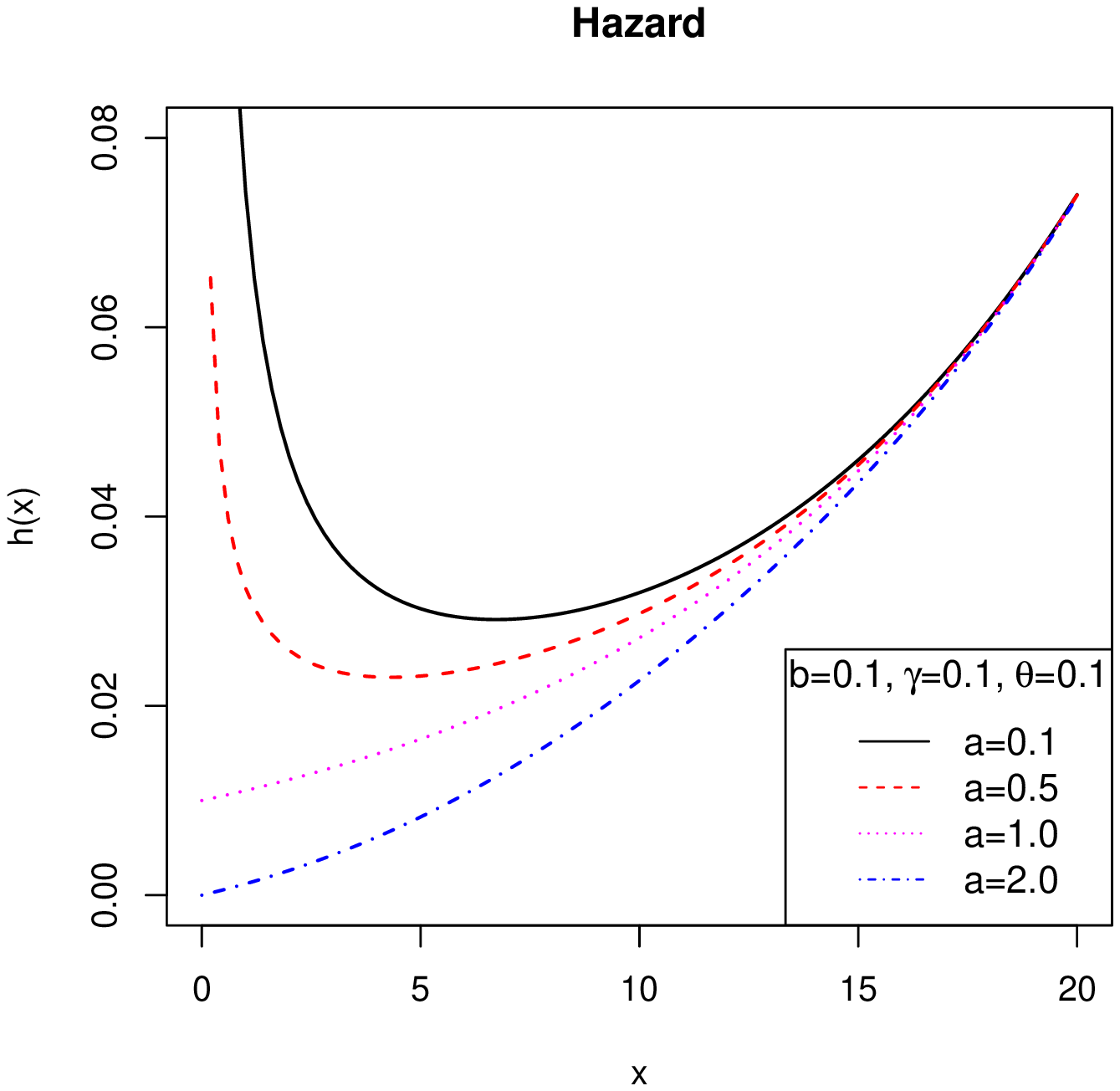}
\includegraphics[scale=0.27]{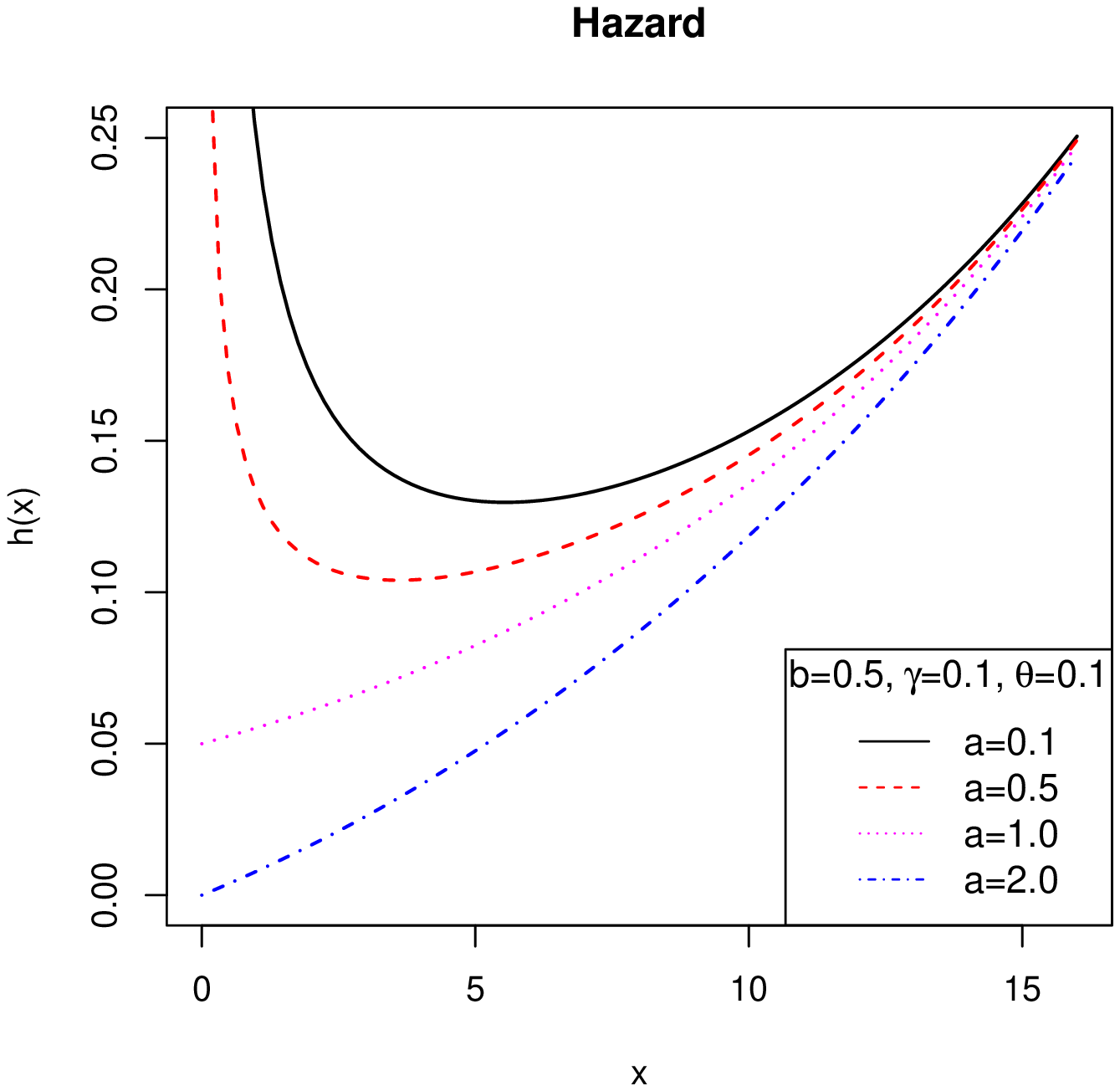}
\includegraphics[scale=0.27]{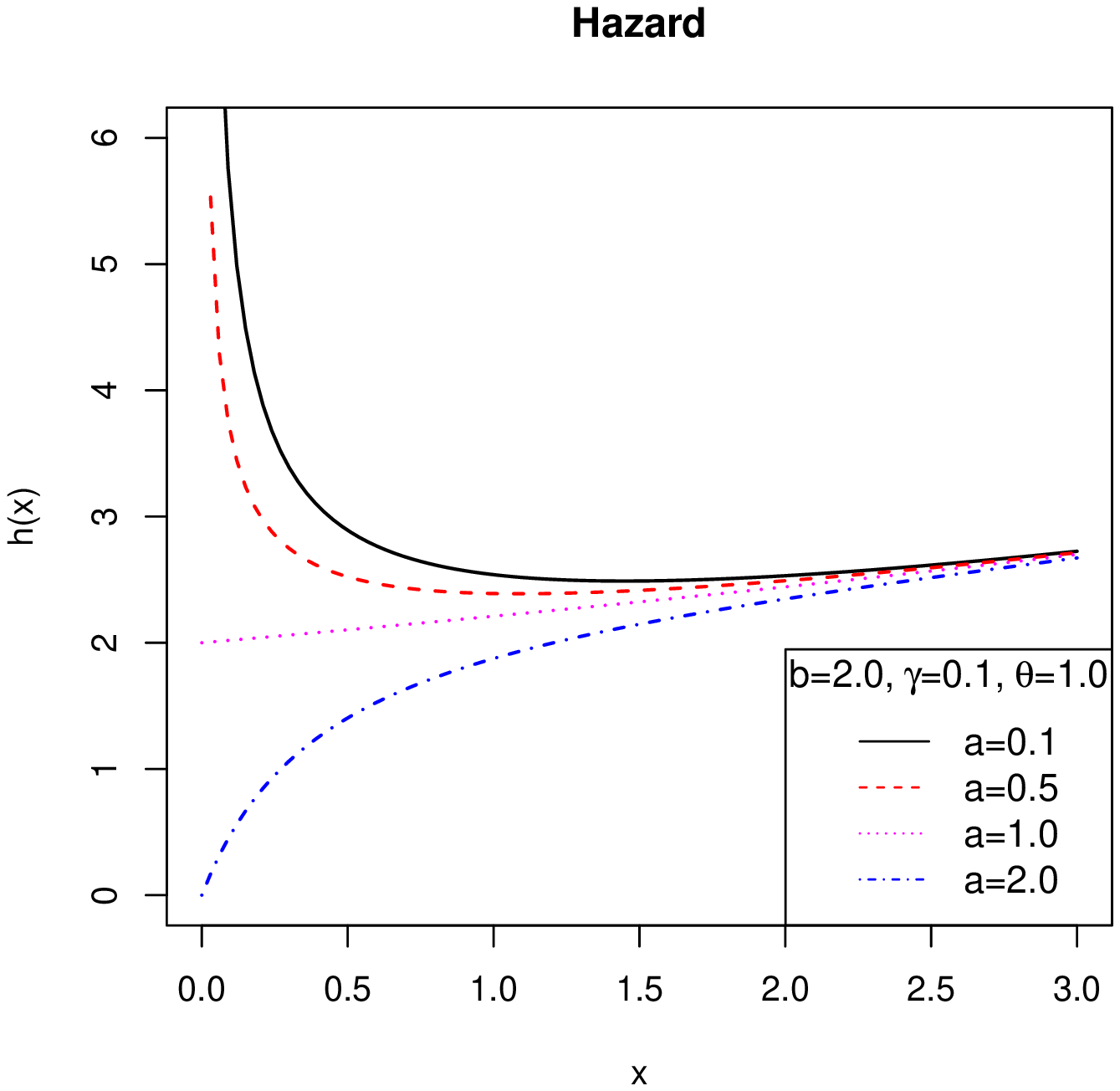}
\includegraphics[scale=0.27]{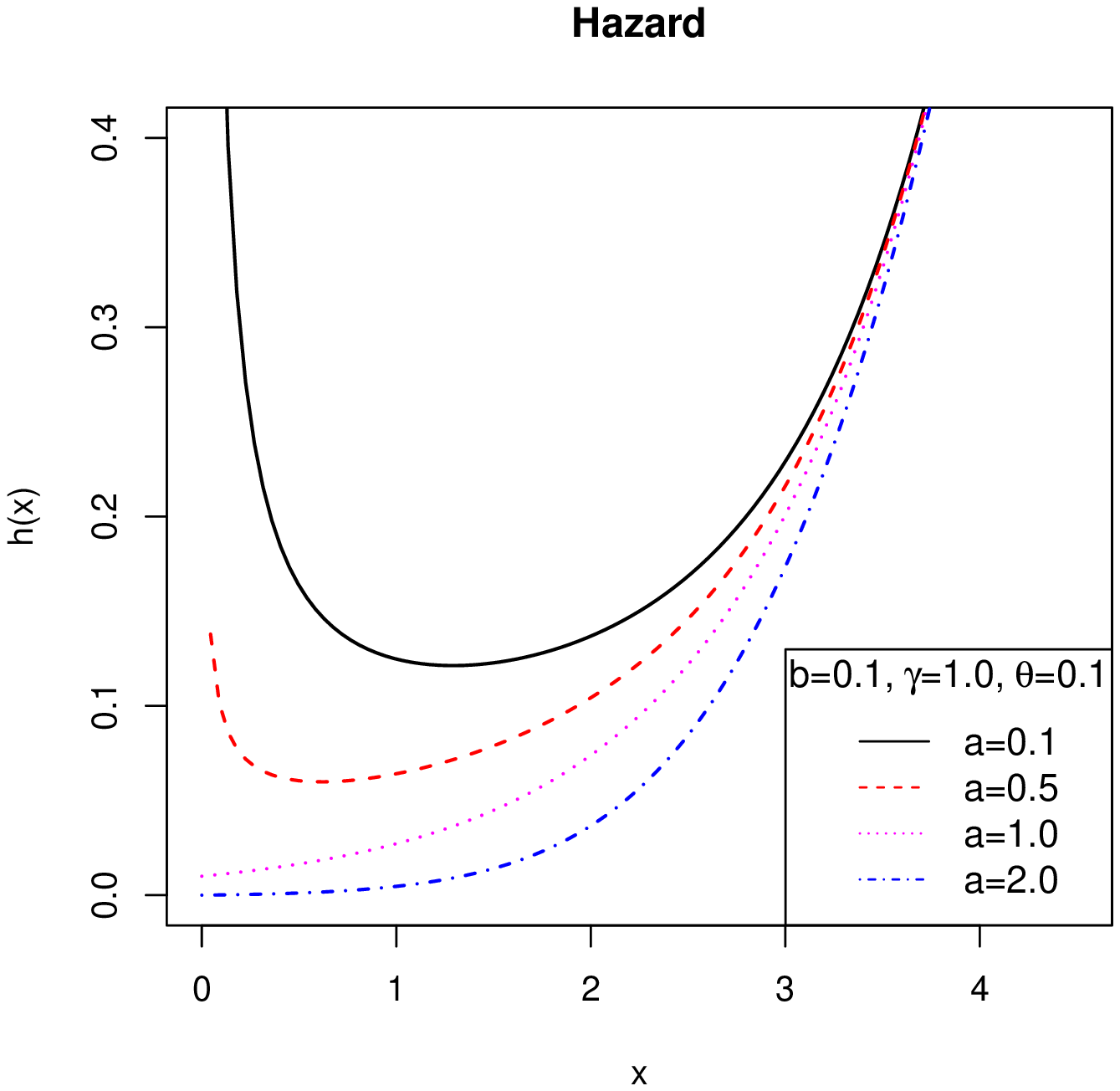}
\includegraphics[scale=0.27]{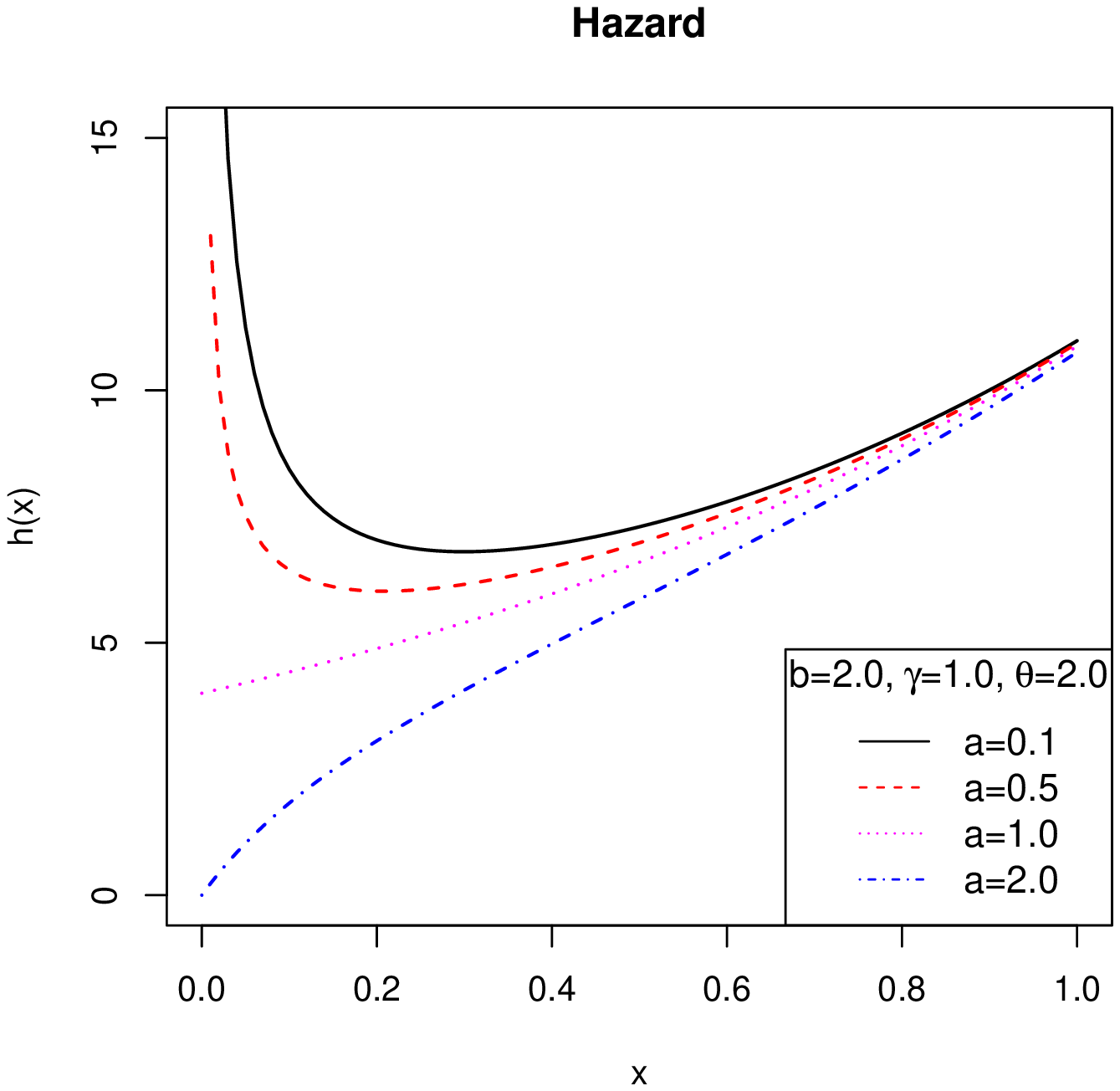}
\includegraphics[scale=0.27]{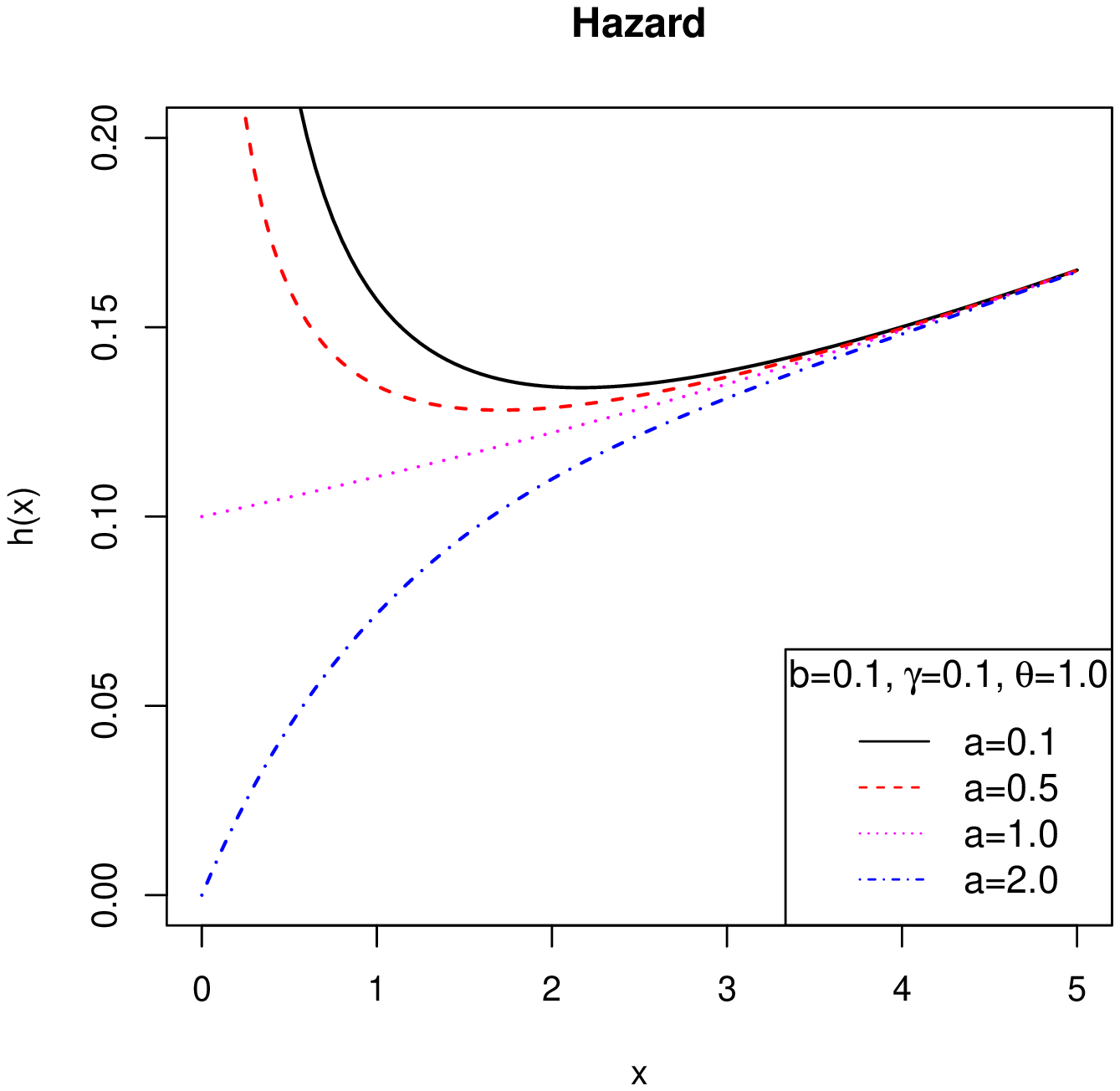}
\caption[]{Plots of hrf of BG for different values of parameters.}\label{fig.hGB}
\end{figure}

Some well-known distributions are special cases of the BG distribution:
 \begin{itemize}
   \item[1.] If $\alpha=1$, $\beta=1$, $\gamma\rightarrow0$, then we get the E distribution.
   \item[2.] If $\beta=1$, $\gamma\rightarrow0$, then we get the GE distribution  which is introduced by \citeasnoun{gu-ku-99} 
   \item[3.]  If $\beta=1$, then we get the GG distribution  which is introduced by \citeasnoun{el-al-al-13}.
   \item[4.]  If $\alpha=1$, $\beta=1$, then we get the G distribution.
   \item[5.]  If $\gamma\rightarrow0$, then we get the BE which is introduced by \citeasnoun{na-ko-06}.
  \end{itemize}

If the random variable $X$ has BG distribution, then it has the following properties:
  \begin{itemize}
   \item[1.] the random variable
   $$Y=1-e^{-\frac{\theta}{\gamma}(e^{\gamma X}-1)},$$
 satisfies the beta distribution with parameters $\alpha$ and $\beta$. Therefore,
  $$T=\frac{\theta}{\gamma}(e^{\gamma X}-1),$$
  satisfies the BE distribution with parameters 1, $\alpha$ and $\beta$  $(BE(1,\alpha,\beta))$.
   \item[2.] If $\alpha=i$ and $\beta=n-i$, where $i$ and $n$ are positive integer values, then the $f(x)$ is the
density function of $i$th order statistic of G  distribution.
   \item[3.]  If $V$ follows Beta distribution with parameters $\alpha$  and $\beta$, then
   $$X=G^{-1}(V)=\frac{1}{\gamma}\log\left(1- \frac{\gamma}{\theta}\log(1-V)\right),$$ follows BG distribution. This result helps in simulating data from the BG  distribution.

 \end{itemize}

For checking the consistency of the simulating data set form BG distribution, the histogram for a generated data set with size 100 and the exact BG density with parameters $\theta=0.1$ and $\gamma=1.0$ , $\alpha=0.1$, and $\beta=0.1$,  are displayed in Figure \ref{Fig.gd} (left). Also, the empirical
distribution function and the exact distribution function is given in Figure \ref{Fig.gd} (right).

\begin{figure}[ht]
\centering
\includegraphics[scale=0.4]{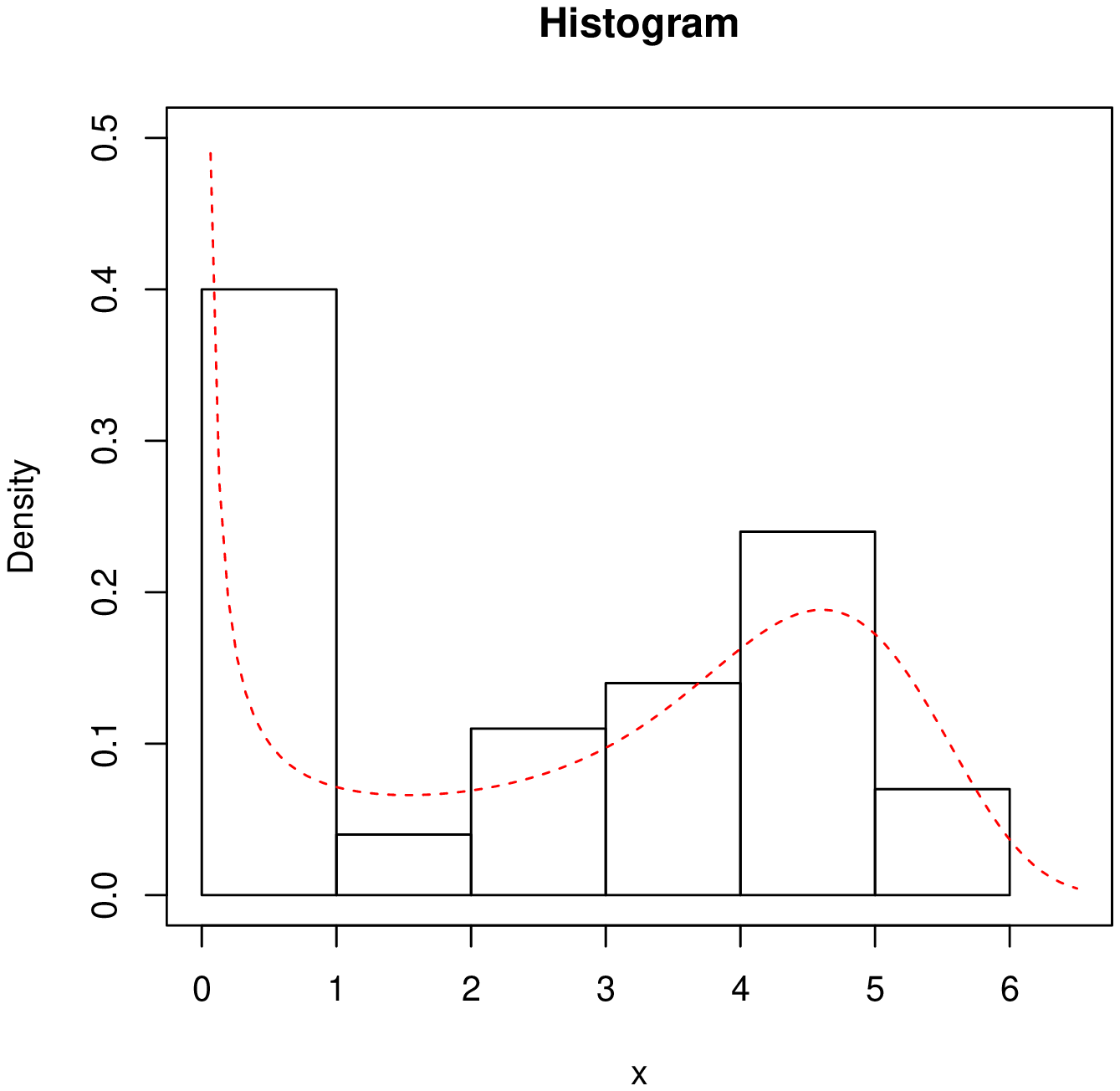}
\includegraphics[scale=0.4]{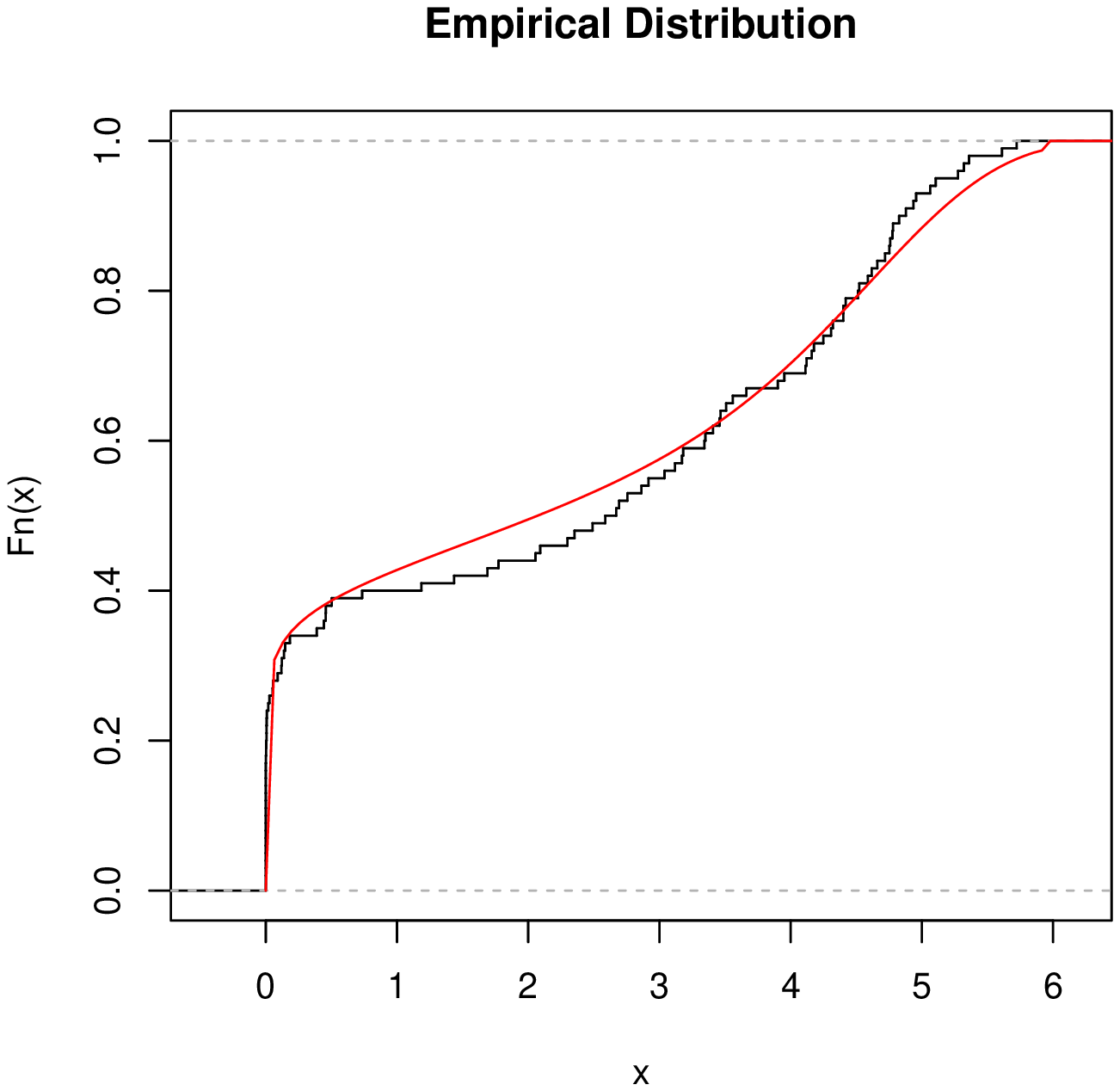}

\caption[]{The histogram of a generated data set with size 100 and the exact GPS  density (left) and the empirical distribution
function and exact distribution function (right). } \label{Fig.gd}
\end{figure}

\section{Some extensions and properties}
\label{sec.pr}
Here, we present some representations of the cdf, pdf, $k$th  moment and moment generating function of BG distribution. Also, we provide expressions for the order statistics, Shannon entropy and quantile measure of this distribution. The mathematical relation given below will be useful in this section. If $\beta$ is a positive real non-integer and $|z|<1$, then  \citeaffixed[p. 25]{gr-ry-07}{}
$$
(1-z)^{\beta-1}=\sum\limits_{j=0}^{\infty} w_{j}z^{j},
$$
and if $\beta$ is a positive real integer, then the upper of the this summation stops at $\beta-1$, where
$$
w_{j}=\frac{(-1)^{j}\Gamma(\beta)}{\Gamma(\beta-j)\Gamma(j+1)}.
$$

\begin{proposition}
We can express (\ref{eq.FG}) as a mixture of distribution function of GG  distributions
as follows:
$$
F(x)=\sum\limits_{j=0}^{\infty} p_{j}[G(x)]^{\alpha+j}=\sum\limits_{j=0}^{\infty} p_{j}G_{j}(x),
$$
where $p_{j}=\frac{(-1)^{j}\Gamma(\alpha+\beta)}{\Gamma(\alpha)\Gamma(\beta-j)\Gamma(j+1)(\alpha+j)}$ and $G_{j}(x)=(G(x))^{\alpha+j}$ is distribution function of a random variable which has a GG  distribution with parameters $\theta$, $\gamma$, and $\alpha+j$. Also, we can write
\begin{eqnarray}
G(x)^{\alpha+j}
&=&\sum\limits_{k=0}^{\infty}(-1)^{k}\binom{\alpha+j}{k}(1-G(x))^{k} \nonumber\\
&=&\sum\limits_{r=0}^{\infty}\sum\limits_{k=r}^{\infty}(-1)^{k+r}\binom{\alpha+j}{k}\binom{k}{r}G(x)^{r},
\end{eqnarray}
and
\begin{eqnarray}
F(x)=\sum\limits_{j=0}^{\infty}\sum\limits_{r=0}^{\infty} \sum\limits_{k=r}^{\infty}p_{j}(-1)^{k+r}\binom{\alpha+j}{k}\binom{k}{r}G(x)^{r}
=\sum\limits_{r=0}^{\infty}b_{r}G(x)^{r},
\end{eqnarray}
where $b_{r}=\sum\limits_{j=0}^{\infty}\sum\limits_{k=r}^{\infty}p_{j}(-1)^{k+r}\binom{\alpha+j}{k}\binom{k}{r}$.
\end{proposition}

\begin{proposition}
We can express (\ref{eq.fBG}) as a mixture of density function of GG  distributions as
follows:
$$
f(x)=\sum\limits_{j=0}^{\infty} p_{j}(\alpha+j)g(x)[G(x)]^{\alpha+j-1}=\sum\limits_{j=0}^{\infty} p_{j}g_{j}(x),
$$
where $g_{j}(x)$ is density function of a random variable which has a GG  distribution with parameters $\theta$, $\gamma$, and $\alpha+j$.
\end{proposition}

\begin{proposition}
The cdf can be expressed in terms of the hypergeometric function and the incomplete beta
function ratio  
\citeaffixed{co-na-11}{see}
 in the following way:
$$
F(x)=\frac{(G(x))^{\alpha}}{\alpha B(\alpha,\beta)}\   _{2}F_{1}(\alpha,1-\beta;\alpha+1;G(x)),
$$
where $_{2}F_{1}(a,b;c;z)=\sum\limits_{k=0}^{\infty}\frac{((a)_{k}(b)_{k})}{((c)_{k}k!)} z^{k}$ and $(a)_{k}=a(a+1)...(a+k-1)$.
\end{proposition}

\begin{proposition}
The $k$th moment of BG  distribution can be expressed as a mixture of the $k$th moment of
GG distributions as follows:
\begin{eqnarray}
E(X^{k})
&=&\int_{0}^{\infty}x^{k}\sum\limits_{j=0}^{\infty}p_{j}(\alpha+j)g(x)[G(x)]^{\alpha+j-1}
=\sum\limits_{j=0}^{\infty}p_{j}E(X_{j}^{k}),
\end{eqnarray}
where
\begin{eqnarray*}
E[X_{j}^{k}]=u_{jk}\sum\limits_{i=0}^{\infty}\sum\limits_{r=0}^{\infty}\binom{\alpha+j-1}{i}\frac{(-1)^{i+r}}{\Gamma(r+1)}e^{\frac{\theta}{\gamma}(i+1)}[\frac{\theta}{\gamma}(i+1)]^{r}[\frac{-1}{\gamma(k+1)}]^{s
+1},
\end{eqnarray*}
$u_{jk}=(\alpha+j)\theta\Gamma(k+1)$ and  $g_{j}(x)$ is density function of a random variable $X_{j}$ which has a GG distribution
with parameters $\theta$, $\gamma$, and $\alpha+j$.
\end{proposition}

\begin{proposition}
The moment generating function of  BG  distribution can be expressed as a mixture of  moment  generating function of
GG distributions as follows:
\begin{eqnarray}
M_{X}(t)
&=&\int_{0}^{\infty}e^{tx}\sum\limits_{j=0}^{\infty}p_{j}(\alpha+j)g(x)[G(x)]^{\alpha+j-1}
=\sum\limits_{j=0}^{\infty}p_{j}M_{X_{j}}(t),
\end{eqnarray}
where
\begin{eqnarray*}
M_{X_{j}}(t)=\frac{(\alpha+j)\theta}{\gamma}\sum\limits_{i=0}^{\infty}\sum\limits_{k=0}^{\infty}(-1)^{i}\binom{\alpha+j-1}{i}\binom{\frac{t}
{\gamma}}{k}\frac{\Gamma(k+1)}{[\frac{(i+1)\theta}{\gamma}]^{k+1}},
\end{eqnarray*}
and  $g_{j}(x)$ is density function of a random variable $X_{j}$ which has a GG distribution
with parameters $\theta$, $\gamma$, and $\alpha+j$.
\end{proposition}

\subsection{Order statistics}
Moments of order statistics play an important role in quality control testing and reliability. For example, if the reliability
of an item is high, the duration of an "all items fail" life–test can be too expensive in both time and money.
 Therefore, a practitioner needs to predict the failure of future items based on the times of
a few early failures. These predictions are often based on moments of order statistics.

Let $X_{1},X_{2},...,X_{n}$  be a random sample of size $n$ from $BG(\theta,\gamma,\alpha,\beta)$. Then the pdf and cdf of the $i$th order
statistic, say $X_{i:n}$, are given by
\begin{eqnarray}\label{eq.fin}
f_{i:n}(x)
&=&\frac{1}{B(i,n-i+1)}\sum\limits_{m=0}^{n-i}(-1)^{m}\binom{n-i}{m}f(x)F^{i+m-1}(x),
\end{eqnarray}
and
\begin{eqnarray}\label{eq.Fin}
F_{i:n}(x)=\int_{0}^{x}f_{i:n}(t)dt=\frac{1}{B(i,n-i+1)}
\sum\limits_{m=0}^{n-i}\frac{(-1)^{m}}{m+i}\binom{n-i}{m}F^{i+m}(x),
\end{eqnarray}
respectively, where  $F^{i+m}(x)=(\sum\limits_{r=0}^{\infty}b_{r}G(x)^{r})^{i+m}$.
Here and henceforth, we use an equation by 
\citeasnoun{gr-ry-07}, page 17, for a power series raised to a positive
integer $n$
\begin{eqnarray}\label{eq.bc}
\left(\sum_{r=0}^\infty b_r\,u^r\right)^n=\sum_{r=0}^\infty
c_{n,r}\,u^r,
\end{eqnarray}where the coefficients $c_{n,r}$ (for
$r=1,2,\ldots$) are easily determined from the recurrence
equation
\begin{equation}\label{eq.cn}
c_{n,r}=(r\,b_0)^{-1}\sum_{m=1}^{r}\,[m\,(n+1)-r]\,b_m\,c_{n,r-m},
\end{equation} where $c_{n,0}=b_0^n$.
 The coefficient $c_{n,r}$ can be calculated from $c_{n,0},\ldots,c_{n,r-1}$ and hence from the quantities $b_0,\ldots,b_{r}$.
 The equations (\ref{eq.fin}) and (\ref{eq.Fin}) can be written as
$$
f_{i:n}(x)=\frac{1}{B(i,n-i+1)}\sum\limits_{m=0}^{n-i}\sum\limits_{r=1}^{\infty}\frac{1}{m+i}(-1)^{m}rc_{i+m,r}g(x)G^{r-1}(x),
$$
and
$$
F_{i:n}(x)=\frac{1}{B(i,n-i+1)}\sum\limits_{m=0}^{n-i}\sum\limits_{r=0}^{\infty}\frac{1}{m+i}(-1)^{m}c_{i+m,r}G^{r}(x).
$$
Therefore,
the $s$th moments of $X_{i:n}$ follows as
\begin{eqnarray*}
E[X_{i:n}^{s}]&=&\frac{1}{B(i,n-i+1)}\sum\limits_{m=0}^{n-i}\sum\limits_{r=1}^{\infty}\frac{1}{m+i}(-1)^{m}rc_{i+m,r}\int_{0}^{+\infty}t^{s}g(t)G^{r-1}(t)dt \nonumber\\
&=&\frac{1}{B(i,n-i+1)}\sum\limits_{m=0}^{n-i}\sum\limits_{r=1}^{\infty}\frac{1}{m+i}(-1)^{m}rc_{i+m,r}\nonumber\\
&\times&
\theta
\Gamma(s+1)\sum\limits_{i_{1}=0}^{\infty}\sum\limits_{i_{2}=0}^{\infty}\binom{r-1}{i_{1}}\frac{(-1)^{i_{1}+i_{2}}}{\Gamma(i_{2}+1)}
e^{\frac{\theta}{\gamma}(i_{1}+1)}[\frac{\theta(i_{1}+1)}{\gamma}]^{i_{2}}[\frac{-1}{\gamma(i_{2}+1)}]^{s+1}.
\end{eqnarray*}

\subsection{Quantile measure }
The quantile function of BG distribution is given by  $$Q(u)=\frac{1}{\gamma}\log(1-\frac{\gamma}{\theta}\log(1-Q_{\alpha,\beta}(u))),$$
where $Q_{\alpha,\beta}(u)$ is the $u$th quantile of beta distribution with parameters $\alpha$ and $\beta$. The effects of the shape parameters $\alpha$ and $\beta$ on
the skewness and kurtosis can be considered based on quantile measures. 
 The Bowley skewness 
\cite{ke-ke-62}
 is one of the earliest skewness measures defined by
$$
{\mathcal{B}}=\frac{Q(\frac{3}{4})+Q(\frac{1}{4})-2Q(\frac{1}{2})}{Q(\frac{3}{4})-Q(\frac{1}{4})}.
$$
Since only the middle two quartiles are considered and the outer
two quartiles are ignored, this adds robustness to the measure.
The Moors kurtosis \cite{moors-88}
is defined as
$$
{\mathcal{M}}=\frac{Q(\frac{3}{8})-Q(\frac{1}{8})+Q(\frac{7}{8})-Q(\frac{5}{8})}{Q(\frac{6}{8})-Q(\frac{2}{8})}.
$$
Clearly, ${\mathcal{M}}> 0$ and there is good concordance with the classical kurtosis measures for some distributions.  These measures are less sensitive to outliers and they exist even for distributions without moments. For the standard normal distribution, these measures are 0 (Bowley) and 1.2331 (Moors).

In Figures \ref{plot.bm1} and \ref{plot.bm2}, we plot the measures $\mathcal{B}$ and $\mathcal{M}$ for some parameter values. These plots indicate that both
measures $\mathcal{B}$ and $\mathcal{M}$ depend on all shape parameters.

\begin{figure}
\centering
\includegraphics[scale=0.42]{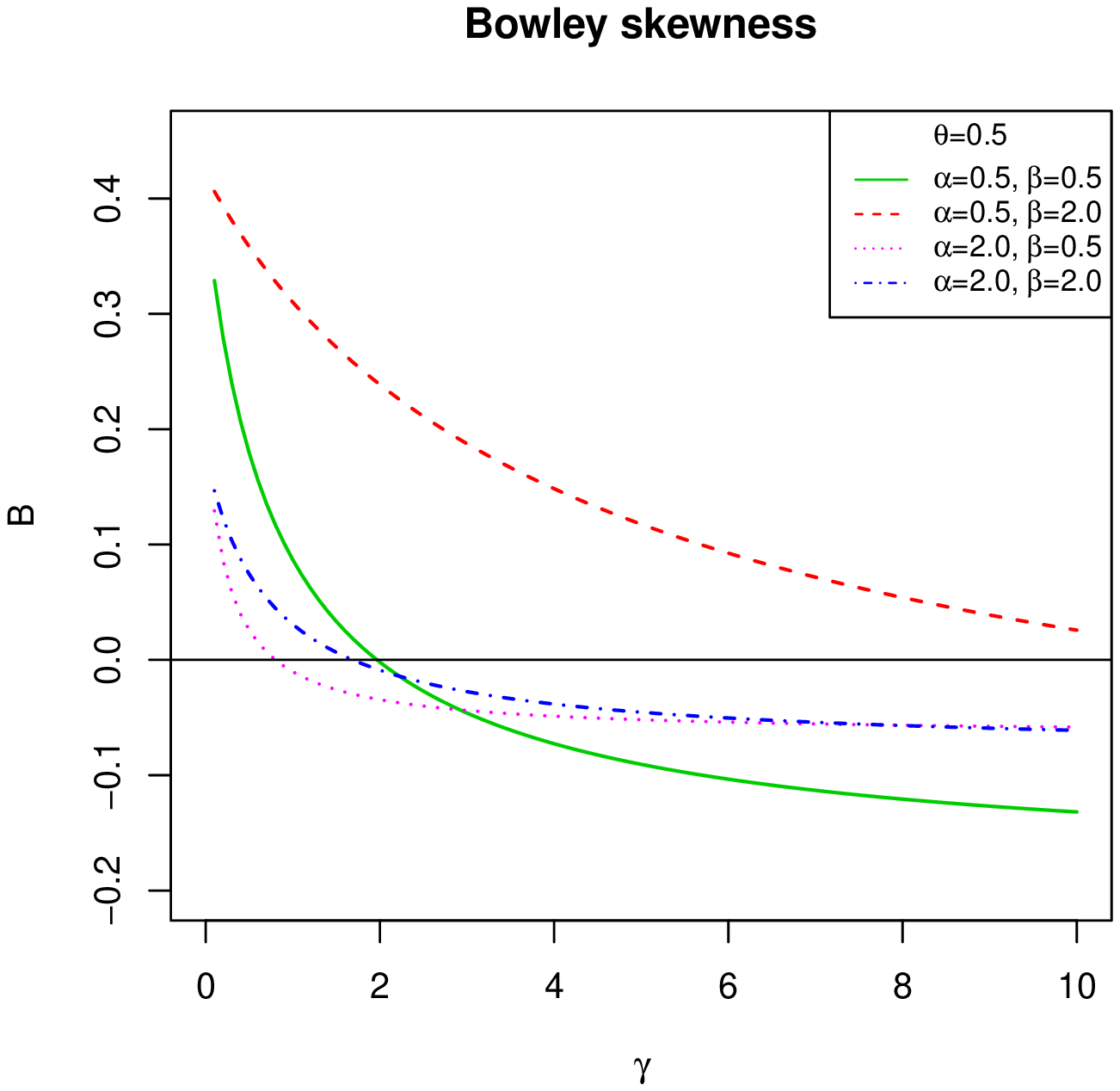}
\includegraphics[scale=0.42]{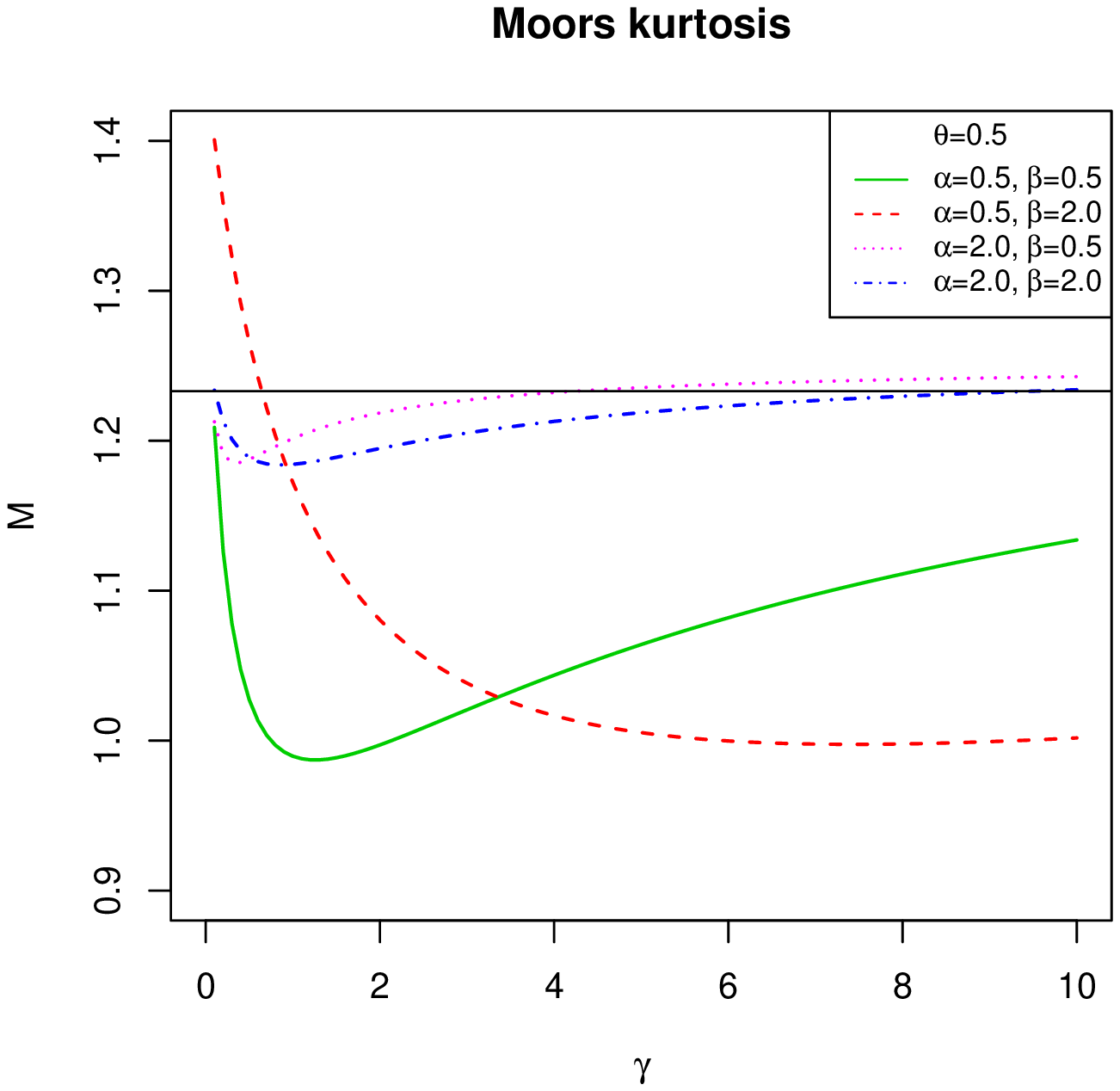}
\caption[]{The Bowley skewness (left) and Moors kurtosis (right) coefficients for the BG distribution as a function of $\gamma$.} \label{plot.bm1}
\end{figure}

\begin{figure}
\centering
\includegraphics[scale=0.42]{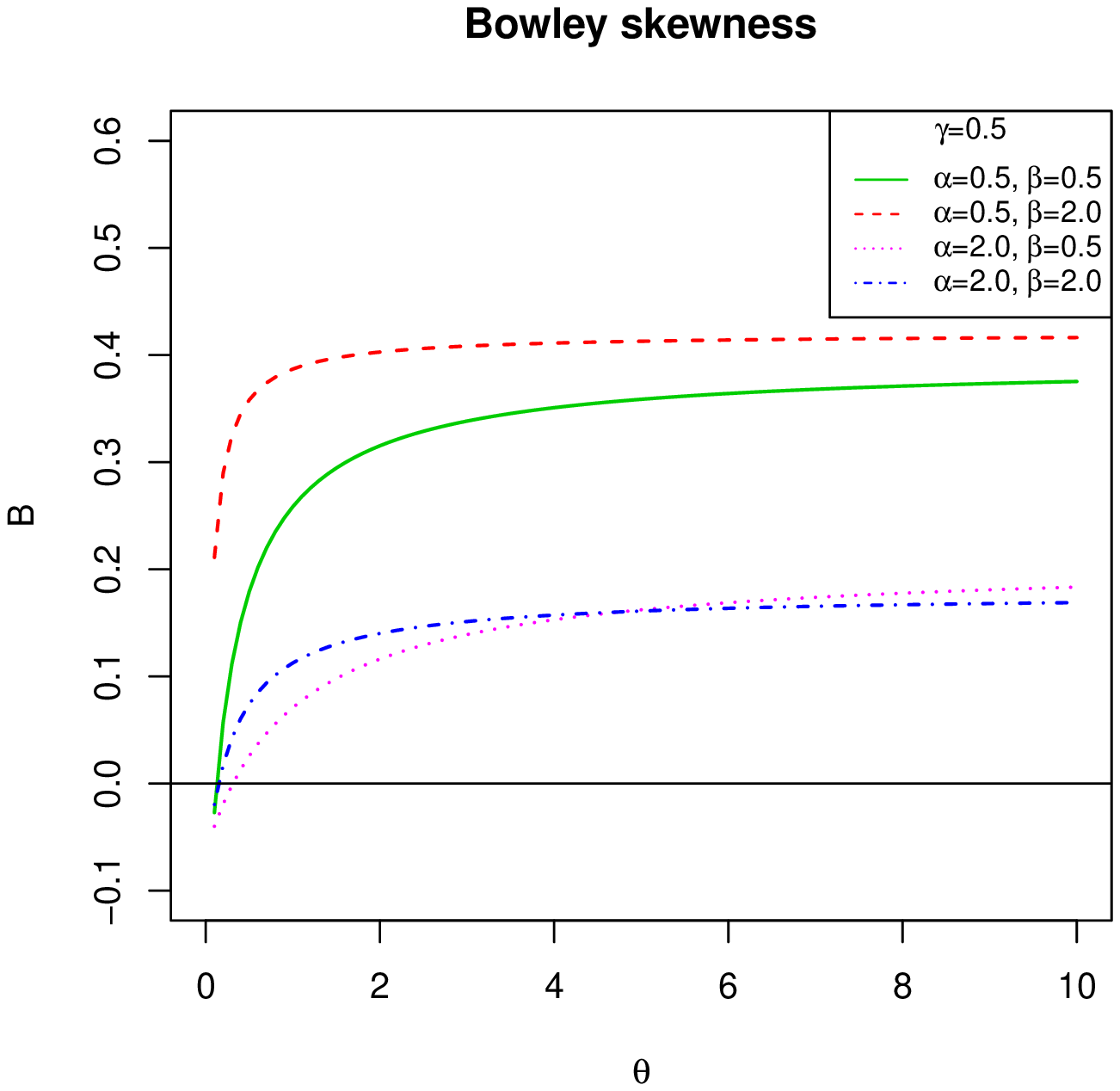}
\includegraphics[scale=0.42]{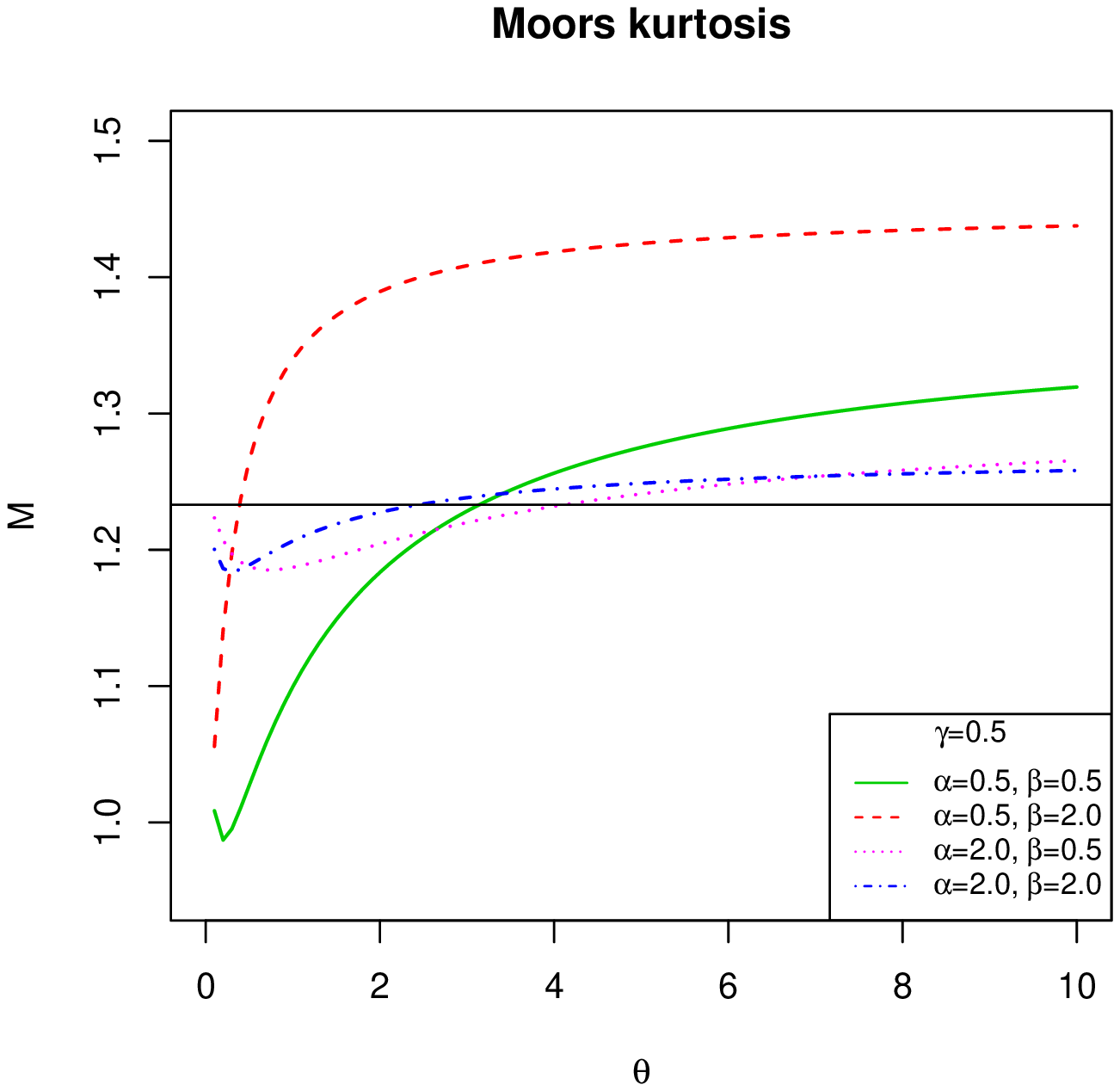}
\caption[]{The Bowley skewness (left) and Moors kurtosis (right) coefficients for the BG distribution as a function of $\theta$.} \label{plot.bm2}
\end{figure}

\subsection{Shannon and R\'{e}nyi entropy }
 If $X$ is a none-negative continuous random variable with pdf $f(x)$, then Shannon's entropy of $X$ is defined  by \citeasnoun{shan-48}
 as
\begin{equation*}
H(f)=E[-\log f(X)]=-\int_{0}^{+\infty} f(x)\log (f(x))dx,
\end{equation*}
and this  is usually referred to as the continuous entropy (or differential entropy). An explicit expression of
Shannon entropy for BG  distribution is obtained as
 \begin{eqnarray}
 H(f)&=&\log(\frac{B(\alpha,\beta)}{\theta})-\frac{\theta \beta}{\gamma}-\gamma E(X)\nonumber\\
 &&+\frac{\theta \beta}{\gamma}M_{X}(\gamma)+(\alpha-1)[\psi(\alpha+\beta)-\psi(\alpha)],
  \end{eqnarray}
  where $\psi(.)$  is a digamma function.

The R\'{e}nyi entropy of order $\lambda$ is defined as
\begin{eqnarray}
H_{\lambda}(f)=\frac{1}{1-\lambda}\log\int_{-\infty}^{+\infty} f^{\lambda}(x)dx, \;\;\;\; \forall \lambda>0 \ \ (\lambda \neq1),
\end{eqnarray}
where
$$ H(X)={\mathop{\lim}_{\lambda\rightarrow1}H_{\lambda}(X)}=-\int_{-\infty}^{+\infty} f(x)\log f(x)dx,$$
is the Shannon entropy, if both integrals exist.  Finally, an explicit expression of R\'{e}nyi  entropy for BG distribution is obtained as
\begin{eqnarray}
H_{\lambda}(f)&=&-\log({\theta})+\frac{\lambda}{\lambda-1}\log(B(\alpha,\beta))+\frac{1}{1-\lambda}[\log(B(\alpha,(\beta-1)\lambda+1))\nonumber\\
&&+
\log(\sum\limits_{j=1}^{\infty}\sum\limits_{k=0}^{j}(-1)^{k}\binom{\lambda-1}{j}\binom{j}{k}
(\frac{\gamma}{\theta})^{j}\frac{\Gamma(j+1)}{(j+1)^{k-1+(\beta-1)\lambda}})].
\end{eqnarray}

\section{Estimation and inference}
\label{sec.es}
In this section, we determine the maximum-likelihood estimates (MLEs) of the parameters of the BG distribution from a complete sample.
Consider $X_{1},...,X_{n}$ is a random sample from BG distribution. The log-likelihood function
for the vector of parameters $\Theta=(\theta,\gamma,\alpha,\beta)$ can be written as
\begin{eqnarray}\label{eq.lik}
l_{n}&=&l_{n}(\Theta)\\
&=&
n\log(\theta)-n\log(B(\alpha,\beta))+n\gamma\bar{x}-
\beta\theta\sum\limits_{i=1}^{n}\log(t_{i})
+(\alpha-1)\sum\limits_{i=1}^{n}\log(1-t_{i}^\theta),\nonumber
\end{eqnarray}
where $\bar{x}=n^{-1}\sum\limits_{i=1}^{n} x_{i}$ and $t_{i}=e^{\frac{-1}{\gamma}(e^{\gamma x_{i}}-1)}$.
The log-likelihood can be maximized either directly or by solving the nonlinear likelihood equations obtained by differentiating (\ref{eq.lik}). The components of the score
vector $U(\Theta)$ are given by
\begin{eqnarray*}
&&U_{\alpha}(\Theta)=\frac{\partial l_{n}}{\partial \alpha}=n\psi(\alpha+\beta)-n\psi(\alpha)+\sum\limits_{i=1}^{n}\log(1-t_{i}^\theta),\\
&&U_{\beta}(\Theta)=\frac{\partial l_{n}}{\partial \beta}=n\psi(\alpha+\beta)-n\psi(\beta)-\theta\sum\limits_{i=1}^{n}\log(t_{i}),\\
&&U_{\theta}(\Theta)=\frac{\partial l_{n}}{\partial \theta}=\frac{n}{\theta}-
\beta\sum\limits_{i=1}^{n}\log(t_i)-
(\alpha-1)\sum\limits_{i=1}^{n}\frac{t_{i}^\theta\log(t_i)}{1-t_{i}^\theta},\\
&&U_{\gamma}(\Theta)=\frac{\partial l_{n}}{\partial \gamma}=n\bar{x}-\beta\theta\sum\limits_{i=1}^{n}d_i
-\theta(\alpha-1)\sum\limits_{i=1}^{n}\frac{d_it_{i}^{\theta}}{1-t_{i}^\theta},
\end{eqnarray*}
where $\psi(.)$ is the digamma function, 
and $d_i=\frac{\partial \log(t_{i})}{\partial \gamma}=\frac{1}{\gamma}(-\log(t_i)+\gamma x_i\log(t_i)-x_i)$.

For interval estimation and hypothesis tests on the model parameters, we require the observed
information matrix. The $4\times4$ unit observed information matrix $J = J_{n}(\Theta)$ is obtained as
\[J=-\left[ \begin{array}{cccc}
J_{\alpha \alpha }& J_{\alpha \beta } & J_{\alpha \theta } & J_{\alpha \gamma } \\
J_{\alpha\beta}   & J_{\beta \beta }  & J_{\beta  \theta } & J_{\beta  \gamma } \\
J_{\alpha \theta} & J_{\beta \theta } & J_{\theta \theta } & J_{\theta \gamma } \\
J_{\alpha \gamma} & J_{\beta \gamma } & J_{\theta \gamma } & J_{\gamma \gamma }
\end{array} \right],\]
where the expressions for the elements of $J$ are
 \begin{align*}
&J_{\alpha\alpha}= \frac{\partial^{2}l_{n}}{\partial \alpha^{2}}=n\psi'(\alpha+\beta)-n\psi'(\alpha), \ \ \ \ \ \ \
 J_{\alpha\beta}= \frac{\partial^{2}l_{n}}{\partial \alpha \partial \beta }=\frac{\partial^{2}l_{n}}{\partial \beta \partial\alpha  }=n\psi'(\alpha+\beta),
 \\&
 J_{\alpha\theta}= \frac{\partial^{2}l_{n}}{\partial \alpha \partial \theta }=\frac{\partial^{2}l_{n}}{\partial \theta \partial\alpha}=\sum\limits_{i=1}^{n}\frac{t_{i}^\theta\log(t_i)}{1-t_{i}^\theta},
 \ \ \ \ \
 J_{\alpha\gamma}= \frac{\partial^{2}l_{n}}{\partial \alpha \partial \gamma }=\frac{\partial^{2}l_{n}}{\partial \gamma \partial\alpha}=-\theta\sum\limits_{i=1}^{n}\frac{d_it_{i}^{\theta}}{1-t_{i}^\theta},\\
 &J_{\beta\beta}= \frac{\partial^{2}l_{n}}{\partial \beta^{2}}=n\psi'(\alpha+\beta)-n\psi'(\beta),
 \ \ \ \ \ \ \
J_{\beta\theta}= \frac{\partial^{2}l_{n}}{\partial \beta \partial \theta }=\frac{\partial^{2}l_{n}}{\partial \theta \partial\beta}=-\sum\limits_{i=1}^{n}\log(t_{i}),\\
&J_{\beta\gamma}= \frac{\partial^{2}l_{n}}{\partial \beta \partial \gamma }=\frac{\partial^{2}l_{n}}{\partial \gamma \partial\beta}=
-\theta\sum\limits_{i=1}^{n}d_i,
\ \ \
J_{\theta\theta}= \frac{\partial^{2}l_{n}}{\partial \theta^{2}}=-\frac{n}{\theta^{2}}+\theta(\alpha-1)\sum\limits_{i=1}^{n}\frac{t_{i}^{\theta}(\log(t_i))^2}{(1-t_{i}^\theta)^2},
\\
&
J_{\theta\gamma}= \frac{\partial^{2}l_{n}}{\partial \theta \partial \gamma }=\frac{\partial^{2}l_{n}}{\partial \gamma \partial\theta}=
-\beta\sum\limits_{i=1}^{n}d_i
-(\alpha-1)\sum\limits_{i=1}^{n}\frac{d_it_{i}^{\theta}}{1-t_{i}^\theta}\left(\theta\log(t_i)+1+\frac{\theta t_i^\theta \log(t_i)}{1-t_{i}^\theta}\right),
\\
&
J_{\gamma\gamma}=\frac{\partial^{2}l_{n}}{\partial \gamma^2 }=-\beta\theta\sum\limits_{i=1}^{n}q_i
-\theta(\alpha-1)\sum\limits_{i=1}^{n}\frac{t_{i}^{\theta}}{1-t_{i}^\theta}(q_i+\theta d_i^2)-\theta^2(\alpha-1)\sum\limits_{i=1}^{n}\frac{d_i^2t_{i}^{2\theta}}{1-t_{i}^{\theta}},
 \end{align*}
where
$q_i=\frac{\partial d_i}{\partial \gamma}=d_i(x_i-\frac{2}{\gamma})+\frac{x_i}{\gamma}\log(t_i)$.

\section{Simulation studies}
\label{sec.si}
In this section, we performed a simulation study in order to investigate the proposed estimator of parameters based on the proposed MLE method. We generate 10,000 data set with size $n$ from the BG distribution with parameters $a$, $b$, $\theta$, and $\gamma$, and compute the MLE's of the parameters.  We assess the accuracy of the approximation of the standard error of the MLE's determined though the Fisher information matrix and variance of the estimated parameters.
Table \ref{table.sim} show the results for the BG distribution. From these results, we can conclude that:
\begin{enumerate}
  \item[i.] the differences between the average estimates and the true values are almost small,
  \item[ii.]  the MLE's converge to true value in all cases when the sample size increases,
  \item[iii.] the standard errors of the MLEs decrease when the sample size increases.
\end{enumerate}
From these simulation, we can conclude that estimation of parameters using the MLE are satisfactory.

\begin{sidewaystable}

\caption{The simulated MLE's and mean of the standard errors for BG
distribution based on information matrix and variance of MLE's.}\label{table.sim}
{\small
\begin{center}
\begin{tabular}{|c|cccc|cccc|cccc|cccc|} \hline
   &\multicolumn{4}{|c|}{parameters}&\multicolumn{4}{|c|}{estimations}&\multicolumn{4}{|c|}{simulated}&\multicolumn{4}{|c|}{Information matrix}\\ \hline
   $n$& $\alpha$ &$\beta$&$\theta$ &$\gamma$ &$\hat\alpha$ &$\hat\beta$&$\hat\theta$ &$\hat\gamma$&
$s.e.(\hat\alpha)$ &$s.e.(\hat\beta)$&$s.e.(\hat\theta)$ &$s.e.(\hat\gamma)$&
$s.e.(\hat\alpha)$ &$s.e.(\hat\beta)$&$s.e.(\hat\theta)$ &$s.e.(\hat\gamma)$ \\  \hline
30 &0.5 & 0.5 & 0.5 & 0.5 & 0.515 & 0.739 & 0.555 & 0.616 & 0.171 & 0.646 & 0.818 & 0.291 & 0.111 & 0.183 & 0.117 & 0.072 \\
   &0.5 & 0.5 & 0.5 & 1.0 & 0.516 & 0.734 & 0.605 & 1.164 & 0.180 & 0.609 & 0.800 & 0.472 & 0.123 & 0.168 & 0.137 & 0.089 \\
   &0.5 & 0.5 & 1.0 & 1.0 & 0.517 & 0.858 & 1.195 & 1.245 & 0.166 & 0.604 & 0.845 & 0.583 & 0.117 & 0.191 & 0.154 & 0.132 \\
   &0.5 & 2.0 & 0.5 & 1.0 & 0.509 & 1.741 & 0.902 & 1.368 & 0.144 & 0.824 & 0.388 & 0.839 & 0.111 & 0.306 & 0.129 & 0.222 \\
   &0.5 & 2.0 & 1.0 & 1.0 & 0.497 & 1.830 & 1.441 & 1.478 & 0.130 & 0.877 & 0.575 & 1.100 & 0.105 & 0.354 & 0.217 & 0.352 \\
   &2.0 & 2.0 & 0.5 & 0.5 & 2.068 & 1.916 & 0.811 & 0.826 & 0.972 & 0.744 & 0.881 & 0.544 & 0.333 & 0.147 & 0.217 & 0.094 \\
\hline

50 &0.5 & 0.5 & 0.5 & 0.5 & 0.506 & 0.612 & 0.526 & 0.567 & 0.125 & 0.600 & 0.900 & 0.214 & 0.088 & 0.143 & 0.111 & 0.054 \\
   &0.5 & 0.5 & 0.5 & 1.0 & 0.509 & 0.729 & 0.556 & 1.092 & 0.131 & 0.633 & 0.864 & 0.342 & 0.085 & 0.137 & 0.096 & 0.077 \\
   &0.5 & 0.5 & 1.0 & 1.0 & 0.512 & 0.829 & 1.104 & 1.149 & 0.126 & 0.598 & 0.941 & 0.421 & 0.086 & 0.147 & 0.123 & 0.104 \\
   &0.5 & 2.0 & 0.5 & 1.0 & 0.507 & 1.896 & 0.888 & 1.183 & 0.111 & 0.795 & 0.345 & 0.631 & 0.081 & 0.244 & 0.082 & 0.179 \\
   &0.5 & 2.0 & 1.0 & 1.0 & 0.501 & 1.932 & 1.202 & 1.202 & 0.101 & 0.856 & 0.515 & 0.873 & 0.080 & 0.288 & 0.147 & 0.282 \\
   &2.0 & 2.0 & 0.5 & 0.5 & 2.201 & 1.991 & 0.678 & 0.659 & 0.845 & 0.734 & 0.856 & 0.350 & 0.372 & 0.114 & 0.150 & 0.060 \\
\hline

100&0.5 & 0.5 & 0.5 & 0.5 & 0.498 & 0.520 & 0.511 & 0.532 & 0.086 & 0.617 & 0.948 & 0.149 & 0.060 & 0.101 & 0.076 & 0.039 \\
   &0.5 & 0.5 & 0.5 & 1.0 & 0.498 & 0.531 & 0.537 & 1.039 & 0.090 & 0.655 & 0.917 & 0.241 & 0.061 & 0.098 & 0.073 & 0.054 \\
   &0.5 & 0.5 & 1.0 & 1.0 & 0.503 & 0.540 & 0.984 & 1.088 & 0.086 & 0.634 & 1.033 & 0.295 & 0.059 & 0.107 & 0.085 & 0.076 \\
   &0.5 & 2.0 & 0.5 & 1.0 & 0.504 & 1.929 & 0.492 & 1.062 & 0.076 & 0.722 & 0.349 & 0.425 & 0.055 & 0.172 & 0.046 & 0.133 \\
   &0.5 & 2.0 & 1.0 & 1.0 & 0.502 & 1.917 & 1.259 & 1.163 & 0.072 & 0.814 & 0.479 & 0.608 & 0.055 & 0.212 & 0.084 & 0.209 \\
   &2.0 & 2.0 & 0.5 & 0.5 & 2.195 & 2.062 & 0.550 & 0.590 & 0.648 & 0.738 & 0.791 & 0.229 & 0.261 & 0.083 & 0.060 & 0.044 \\
\hline
\end{tabular}
\end{center}
}

\end{sidewaystable}

\section{Application of BG to real data set}
\label{sec.ex}

In this section, we perform an application to real data  and demonstrate the superiority of BG distribution as compared to some of
its sub-models.  The data have been obtained from \citeasnoun{Aarset-87}, 
and  widely reported in some literatures \citeaffixed{si-or-co-10}{for example see}. It represents the lifetimes of 50 devices, and also, possess a bathtub-shaped failure rate property. The  numerical evaluations were implemented using R software (nlminb function).



Based on some goodness-of-fit measures, the performace of the BG distribution is quantified and compared with others due to  five literature distributions:  E, GE, BE, G, and GG,  distributions. The MLE's of the unknown parameters (standard errors in parentheses) for these distributions are given in Table \ref{table.EX}. Also, the values of the log-likelihood functions ($-\log(L)$), the  Kolmogorov–Smirnov (K–S) test statistic with its p-value,  the statistics AIC (Akaike Information Criterion), the statistics AICC (Akaike Information Criterion with correction) and BIC (Bayesian Information Criterion) are calculated for the six distributions in order  to verify which distribution fits better to these data. All the computations were done using the R software.

The BG distribution yields the highest value of the log-likelihood function and smallest values of the AIC, AICC and BIC statistics. From the values of these statistics, we can conclude that the BG model is better than the other distributions to fit these data. The plots of the densities (together with the data histogram)  and cumulative distribution functions (with empirical distribution function) are given in Figure \ref{plot.EX}. It is evident that the BG model provides a better fit than the other models. In particular, the histogram of data shows that the BG model provides an excellent fit to these data.

For this data set, we perform the  Likelihood ratio test (LRT)  for testing the following hypotheses:
\begin{itemize}
  \item[1.] $H_0$: E distribution vs. $H_1$: BG distribution
  \item[2.] $H_0$: GE distribution vs. $H_1$: BG distribution
  \item[3.] $H_0$: BE distribution vs. $H_1$: BG distribution
  \item[4.] $H_0$: G distribution vs. $H_1$: BG distribution, or equivalently $H_0$: $(\alpha,\beta)=(1,1)$ vs. $H_1$: $(\alpha,\beta)\neq(1,1)$
  \item[5.] $H_0$: GG distribution vs. $H_1$: BG distribution, or equivalently $H_0$: $\beta=1$ vs. $H_1$: $\beta\neq 1$.
\end{itemize}
The values of LRT statistic and its corresponding p-value for each hypotheses are given in Table \ref{table.EX}. From these results, we can conclude that the null hypotheses are rejected in all situations, and therefore, the BG distribution is an adequate model.

\begin{remark}
\citeasnoun{el-al-al-13} found the following estimations for the parameters of GG distribution:
$$\hat{a}=0.421, \ \ \hat{\theta}=0.00143, \ \  \hat{\gamma}=0.044.$$
 Based on these estimations, the log-likelihood function is equal to $-224.1274$. But we found the following estimations for the parameters of GG distribution:
$$\hat{a}=0.2625, \ \ \hat{\theta}=0.0001, \ \  \hat{\gamma}=0.0828.$$
 Based on these estimations, the log-likelihood function is equal to $-222.2441$. Therefore, the estimations  of \citeasnoun{el-al-al-13} for GG distribution is not the MLE.
\end{remark}

\begin{table}
\begin{center}
\caption{Parameter estimates (with std.), K-S statistic,
\textit{p}-value for K-S, AIC, AICC, BIC, LRT statistic and \textit{p}-value of LRT    for the data set.}\label{table.EX}
{\small
\begin{tabular}{|cl|cccccc|} \hline
&Distribution     & E       & GE      & BE       & G       & GG         & BG      \\ \hline
&$\hat{\alpha}$   & ---     & 0.9021  & 0.5236   & ---     & 0.2625     & 0.2158   \\
&(std.)           & ---     &(0.1349) &(0.1714)  & ---     &(0.0395)    &(0.0392) \\\hline
&$\hat{\beta }$   & ---     & ---     &0.0847    & ---     & ---        &0.2467   \\
&(std.)           & ---     & ---     &(0.0828)  & ---     & ---        &(0.0448)  \\\hline
&$\hat{\theta }$  & 0.0219   & 0.0212   & 0.2352    & 0.0097   & 0.0001      & 0.0003   \\
&(std.)           & (0.0031) &(0.0036)  & (0.2111)  & (0.0029) & (0.0001)    & (0.0001) \\\hline
&$\hat{\gamma}$   & ---     & ---     & ---      & 0.0203   & 0.0828      & 0.0882   \\
& (std.)          & ---     & ---     & ---      & (0.0058) & (0.0031)    &(0.0030)  \\ \hline
&$-\log(L)$             & 241.0896       & 240.3855       & 238.1201      & 235.3308      & 222.2441         & 220.6714   \\
&K-S                    & 0.1911         & 0.1940         & 0.1902        & 0.1696        & 0.1409           & 0.1322     \\
&p-value (K-S)               & 0.0519         & 0.0514         & 0.0538        & 0.1123        & 0.2739           & 0.3456     \\
&AIC                    & 484.1792       & 484.7710       & 482.2400      & 474.6617      & 450.4881         & 449.3437   \\
&AICC                   & 484.2625       & 485.0264       & 482.7617      & 475.1834      & 451.0099         & 450.2326   \\
&BIC                    & 486.0912       & 488.5951       & 487.9760      & 482.3977      & 456.2242         & 456.9918 \\
&LRT                    & 40.8355        & 39.4273        & 34.8962       & 29.3179       & 3.1444           &--- \\
&p-value (LRT)          & 0.0000         & 0.0000         & 0.0000        & 0.0001        &  0.0762          &--- \\
 \hline
\end{tabular}
}
\end{center}
\end{table}

\begin{figure}
\vspace{-0.0cm}
\centering
\includegraphics[scale=0.42]{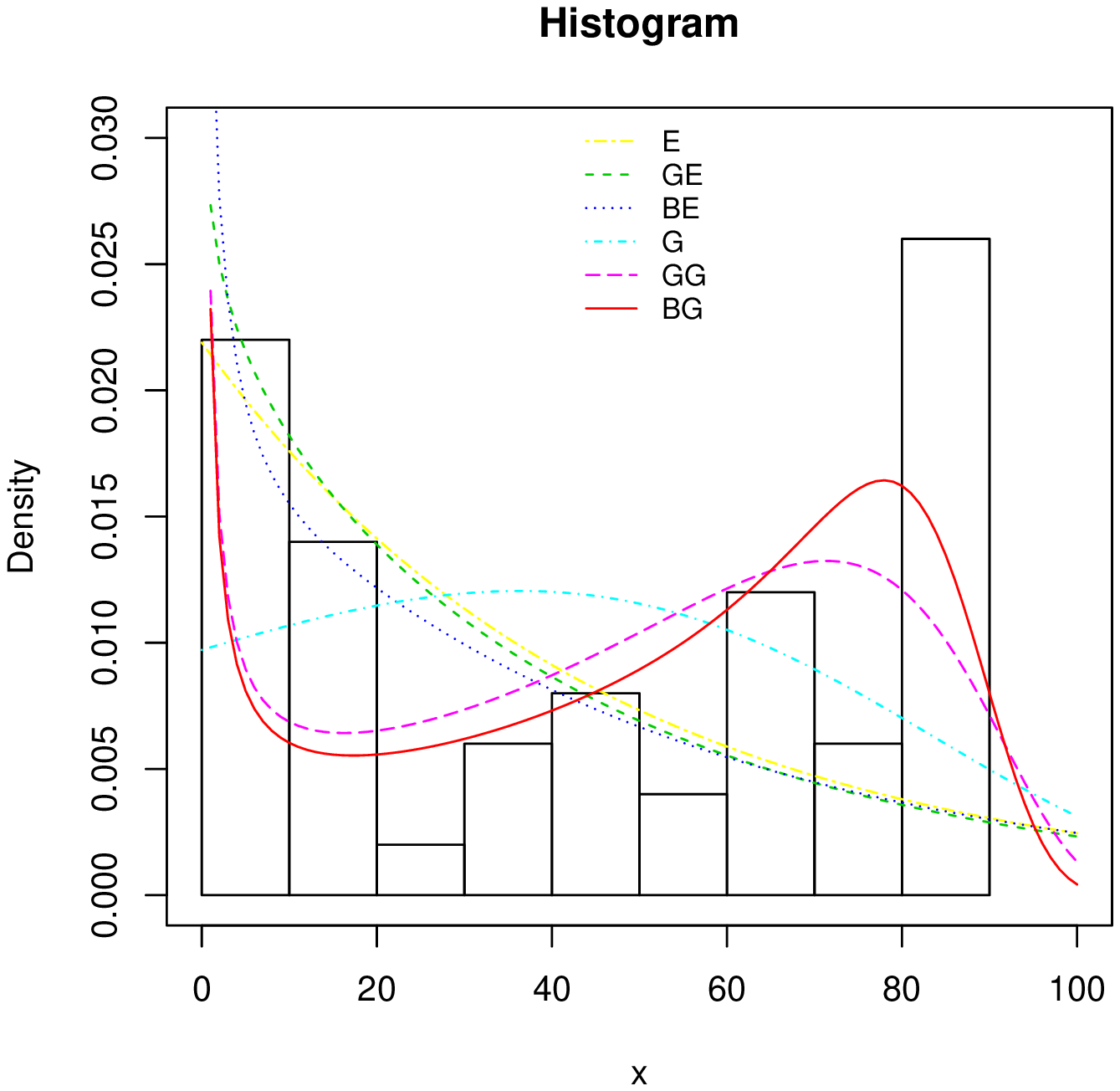}
\includegraphics[scale=0.42]{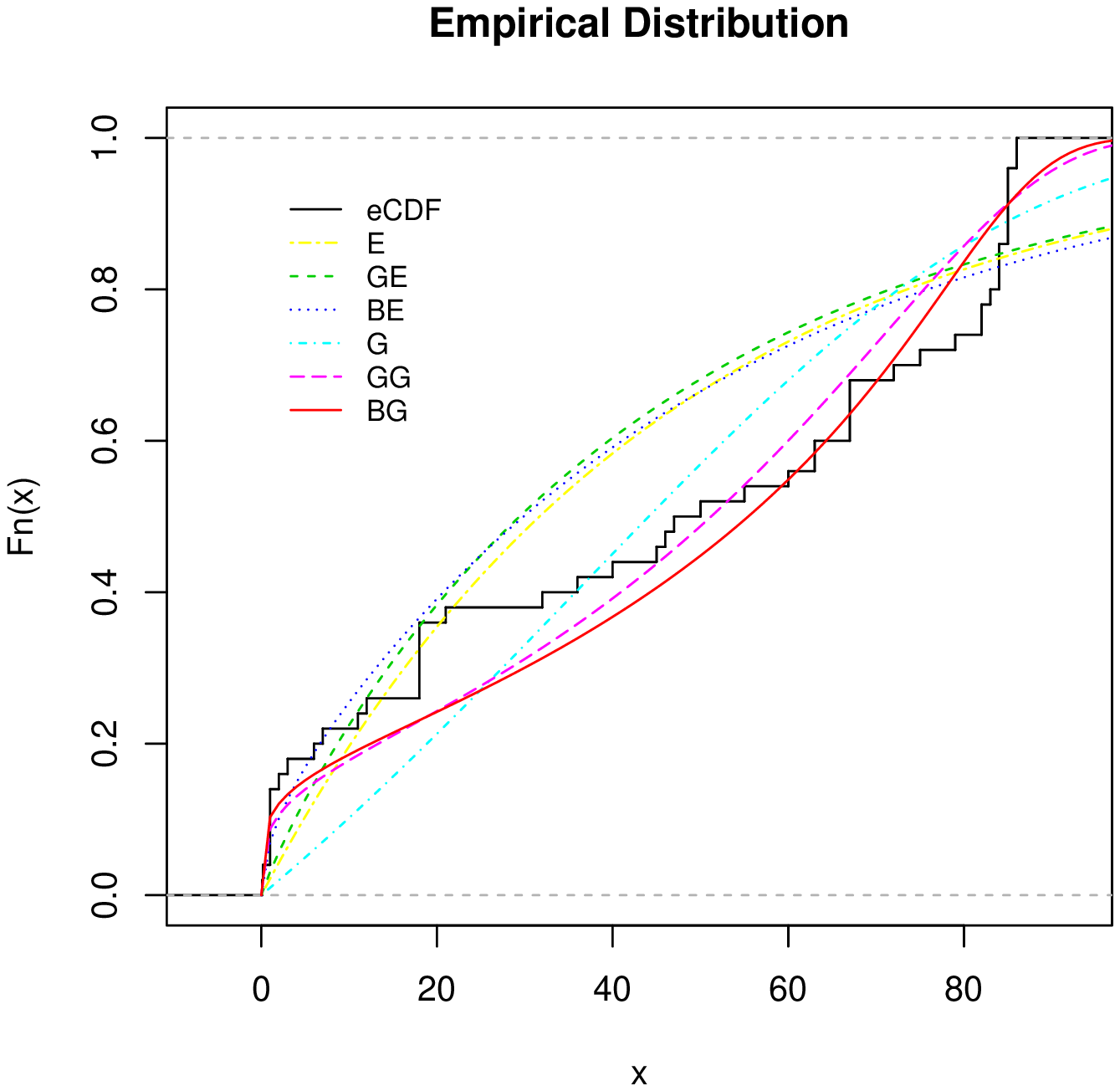}
\caption[]{Plots (density and distribution) of  fitted  E, GE, BE, G,  GG and BG distributions for the data set.}\label{plot.EX}
\end{figure}

\section*{Acknowledgements}

The authors would like to thank the  editor and the anonymous referees for their constructive comments and suggestions that appreciably improved the quality of presentation of this manuscript.

\newpage



\end{document}